\documentclass[11pt]{article}

\usepackage{amsmath,amssymb,amsthm,mathtools}
\usepackage[margin=1in]{geometry}
\usepackage{microtype}
\usepackage{enumitem}
\usepackage{booktabs}
\usepackage[english]{babel}
\usepackage[dvipsnames]{xcolor}
\usepackage{tikz}
\usetikzlibrary{arrows.meta,calc,fit,positioning,shapes.geometric}
\usepackage[font=small,labelfont=bf]{caption}
\usepackage[hidelinks]{hyperref}
\hypersetup{
  pdftitle={Negative Correlation for Forests and the q<1 Random-Cluster Model},
  pdfauthor={R. Altar Ciceksiz and Mohan Ravichandran},
  pdfkeywords={negative correlation, arboreal gas, spanning forest, random-cluster model, Rayleigh difference, effective resistance}
}

\usepackage{comment}
\newtheorem{thm}{Theorem}[section]
\newtheorem{prop}[thm]{Proposition}
\newtheorem{lemma}[thm]{Lemma}
\newtheorem{corollary}[thm]{Corollary}
\newtheorem{conj}[thm]{Conjecture}

\theoremstyle{definition}

\newtheorem{example}[thm]{Example}
\theoremstyle{remark}
\newtheorem{rmk}[thm]{Remark}

\DeclareMathOperator{\Cov}{Cov}
\DeclareMathOperator{\cl}{cl}
\newcommand{\ForestZ}{Z^{\mathcal{F}}}
\newcommand{\RCZ}{Z^{\mathrm{RC}}}
\newcommand{\Rayleigh}{\mathcal R}
\newcommand{\Prob}{\mathbb P}
\newcommand{\Exp}{\mathbb E}
\newcommand{\1}{\mathbf 1}

\setlist{nosep,leftmargin=*}

\definecolor{markedE}{HTML}{B4372E}
\definecolor{markedF}{HTML}{2468A2}
\definecolor{addedEdge}{HTML}{27834A}
\definecolor{softGold}{HTML}{E9D9A7}
\usepackage{xcolor}
\usepackage[normalem]{ulem}

\definecolor{darkgreen}{RGB}{0,100,0}

\tikzset{
  vtx/.style={circle,draw=black,fill=white,line width=.45pt,
    minimum size=4.5mm,inner sep=1pt,font=\scriptsize},
  tinyvtx/.style={circle,draw=black,fill=white,line width=.4pt,
    minimum size=3.2mm,inner sep=.5pt,font=\tiny},
  edgeE/.style={draw=markedE,line width=1.5pt},
  edgeF/.style={draw=markedF,line width=1.5pt},
  added/.style={draw=addedEdge,line width=1.35pt},
  ghost/.style={densely dashed,line width=1.05pt},
  ordinary/.style={draw=black!62,line width=.65pt},
  blob/.style={draw=black!48,rounded corners=5pt,fill=black!3,
    inner sep=4pt},
  flowbox/.style={draw=black!60,rounded corners=3pt,fill=black!2,
    inner sep=5pt,align=center,font=\small},
  figlabel/.style={font=\small\bfseries,anchor=north west}
}

\newcommand{\DeletedGraphFigure}{%
\begin{tikzpicture}[x=1cm,y=1cm]
  \begin{scope}[xshift=0cm]
    \node[font=\scriptsize\bfseries] at (0,1.48) {(a) Both pairs connected};
    \draw[rounded corners=8pt,fill=markedE!7,draw=markedE!35]
      (-1.83,.25) rectangle (1.83,.88);
    \draw[rounded corners=8pt,fill=markedF!7,draw=markedF!35]
      (-1.83,-.88) rectangle (1.83,-.25);
    \coordinate (aa) at (-1.45,.55); \coordinate (ab) at (1.45,.55);
    \coordinate (ac) at (-1.45,-.55); \coordinate (ad) at (1.45,-.55);
    \draw[edgeE,ghost] (aa)--node[above,font=\scriptsize,text=markedE] {$e$} (ab);
    \draw[edgeF,ghost] (ac)--node[below,font=\scriptsize,text=markedF] {$f$} (ad);
    \draw[ordinary] (aa) to[bend left=12] (ab);
    \draw[ordinary] (ac) to[bend right=12] (ad);
    \node[vtx] at (aa) {$a$}; \node[vtx] at (ab) {$b$};
    \node[vtx] at (ac) {$c$}; \node[vtx] at (ad) {$d$};
    \node[font=\scriptsize] at (0,-1.23) {$N_e\cap N_f$};
  \end{scope}

  \begin{scope}[xshift=5.1cm]
    \node[font=\scriptsize\bfseries] at (0,1.48) {(b) First crossing};
    \draw[rounded corners=8pt,fill=softGold!28,draw=black!30]
      (-1.82,-.94) rectangle (-1.08,.94);
    \draw[rounded corners=8pt,fill=black!5,draw=black!30]
      (1.08,-.94) rectangle (1.82,.94);
    \coordinate (ba) at (-1.45,.55); \coordinate (bb) at (1.45,.55);
    \coordinate (bc) at (-1.45,-.55); \coordinate (bd) at (1.45,-.55);
    \draw[edgeE,ghost] (ba)--node[above,font=\scriptsize,text=markedE] {$e$} (bb);
    \draw[edgeF,ghost] (bc)--node[below,font=\scriptsize,text=markedF] {$f$} (bd);
    \draw[ordinary] (ba)--(bc); \draw[ordinary] (bb)--(bd);
    \node[vtx] at (ba) {$a$}; \node[vtx] at (bb) {$b$};
    \node[vtx] at (bc) {$c$}; \node[vtx] at (bd) {$d$};
    \node[font=\scriptsize] at (0,-1.23) {$J_{ef}:\ ac\mid bd$};
  \end{scope}

  \begin{scope}[xshift=10.2cm]
    \node[font=\scriptsize\bfseries] at (0,1.48) {(c) Second crossing};
    \coordinate (ca) at (-1.45,.55); \coordinate (cb) at (1.45,.55);
    \coordinate (cc) at (-1.45,-.55); \coordinate (cd) at (1.45,-.55);
    \draw[edgeE,ghost] (ca)--node[above,font=\scriptsize,text=markedE] {$e$} (cb);
    \draw[edgeF,ghost] (cc)--node[below,font=\scriptsize,text=markedF] {$f$} (cd);
    \draw[ordinary,draw=markedE!65]
      (ca) .. controls (-.75,1.25) and (1.05,.15) .. (cd);
    \draw[ordinary,draw=markedF!65]
      (cb) .. controls (.75,-1.25) and (-1.05,-.15) .. (cc);
    \node[vtx] at (ca) {$a$}; \node[vtx] at (cb) {$b$};
    \node[vtx] at (cc) {$c$}; \node[vtx] at (cd) {$d$};
    \node[font=\scriptsize] at (0,-1.23) {$J_{ef}:\ ad\mid bc$};
  \end{scope}
\end{tikzpicture}%
}

\newcommand{\BridgeTypesFigure}{%
\begin{tikzpicture}[x=1cm,y=1cm]
  \begin{scope}[xshift=0cm]
    \node[font=\scriptsize\bfseries] at (2.4,1.35)
      {(a) Two ordinary bridges};
    \node[blob,minimum width=1.35cm] (OU) at (0,0) {$U$};
    \node[blob,minimum width=1.35cm] (OM) at (2.4,0) {$M$};
    \node[blob,minimum width=1.35cm] (OV) at (4.8,0) {$V$};
    \draw[edgeE] (OU)--node[above,font=\scriptsize,text=markedE] {$e$} (OM);
    \draw[edgeF] (OM)--node[above,font=\scriptsize,text=markedF] {$f$} (OV);
    \node[font=\scriptsize,align=center] at (2.4,-1.5)
      {deleting either marked edge\\disconnects the configuration};
  \end{scope}

  \begin{scope}[xshift=7.35cm]
    \node[font=\scriptsize\bfseries] at (2.4,1.35)
      {(b) A joint-bridge pair};
    \node[blob,minimum width=1.55cm,minimum height=.85cm] (JU) at (0,0) {$U$};
    \node[blob,minimum width=1.55cm,minimum height=.85cm] (JV) at (4.8,0) {$V$};
    \draw[edgeE] (JU.north east) to[bend left=18]
      node[pos=.25,above,font=\scriptsize,text=markedE] {$e$} (JV.north west);
    \draw[edgeF] (JU.south east) to[bend right=18]
      node[pos=.25,below,font=\scriptsize,text=markedF] {$f$} (JV.south west);
    \node[font=\scriptsize,align=center] at (2.4,-1.5)
      {neither edge is a bridge, but deleting\\both separates $U$ from $V$};
  \end{scope}
\end{tikzpicture}%
}

\newcommand{\EquipotentialFigure}{%
\begin{tikzpicture}[x=1cm,y=1cm]
  \filldraw[fill=markedE!7,draw=markedE!45,line width=.8pt]
    (-2.45,0) ellipse (3.05cm and 1.72cm);
  \filldraw[fill=markedF!7,draw=markedF!45,line width=.8pt]
    (2.45,0) ellipse (3.05cm and 1.72cm);
  \node[font=\small\bfseries,text=markedE] at (-3.65,1.24) {$G_e$};
  \node[font=\small\bfseries,text=markedF] at (3.65,1.24) {$G_f$};
  \node[font=\scriptsize,text=markedE,align=center] at (-3.65,-1.16)
    {$\widetilde v_e$ supported here};
  \node[font=\scriptsize,text=markedF,align=center] at (3.65,-1.16)
    {$\widetilde v_f$ supported here};

  \node[vtx] (l1) at (-4.1,.38) {};
  \node[vtx] (l2) at (-3.05,-.4) {};
  \node[vtx] (l3) at (-1.65,.65) {};
  \node[vtx] (l4) at (-1.5,-.67) {};
  \node[vtx] (r1) at (4.1,-.35) {};
  \node[vtx] (r2) at (3.05,.42) {};
  \node[vtx] (r3) at (1.65,.68) {};
  \node[vtx] (r4) at (1.5,-.64) {};
  \node[vtx] (s1) at (0,.72) {};
  \node[vtx] (s2) at (0,0) {};
  \node[vtx] (s3) at (0,-.72) {};

  \draw[ordinary] (l1)--(l3)--(s1) (l2)--(l4)--(s3) (l3)--(l4)--(s2);
  \draw[ordinary] (r1)--(r4)--(s3) (r2)--(r3)--(s1) (r3)--(r4)--(s2);
  \draw[ordinary] (s1)--(s2)--(s3);
  \draw[edgeE] (l1)--node[above left,font=\scriptsize,text=markedE] {$e$} (l2);
  \draw[edgeF] (r1)--node[below right,font=\scriptsize,text=markedF] {$f$} (r2);

  \node[font=\small\bfseries,fill=white,inner sep=2pt] at (0,1.45) {$S$};
  \node[font=\scriptsize,fill=white,inner sep=2pt] at (0,-1.35)
    {$\widetilde v_e=\widetilde v_f=0$};
\end{tikzpicture}%
}

\newcommand{\SymmetricExamplesFigure}{%
\begin{tikzpicture}[x=1cm,y=1cm]
  \begin{scope}[xshift=0cm]
    \node[figlabel] at (-2,2) {(a)};
    \node[font=\scriptsize\bfseries] at (0,1.78) {Complete join};
    \node[vtx,minimum size=8mm] (h1) at (-1.15,.65) {$H_1$};
    \node[vtx,minimum size=8mm] (h2) at (1.15,.65) {$H_2$};
    \node[vtx,minimum size=8mm] (h3) at (-1.15,-.65) {$H_3$};
    \node[vtx,minimum size=8mm] (h4) at (1.15,-.65) {$H_4$};
    \draw[ordinary] (h1)--(h3) (h1)--(h4) (h2)--(h3) (h2)--(h4);
    \draw[edgeE] (h1)--node[above,font=\scriptsize,text=markedE] {$e$} (h2);
    \draw[edgeF] (h3)--node[below,font=\scriptsize,text=markedF] {$f$} (h4);
    \node[font=\tiny,align=center] at (0,-1.35) {untouched parts have a common\\$e$-potential};
  \end{scope}

  \begin{scope}[xshift=5.25cm,scale=.72]
    \node[figlabel] at (-2.6,2.75) {(b)};
    \node[font=\scriptsize\bfseries] at (0,2.47) {$C_3\square C_3$};
    \foreach \i in {0,1,2}{
      \foreach \j in {0,1,2}{
        \node[tinyvtx] (t\i\j) at ({1.35*(\i-1)},{1.25*(\j-1)}) {};
      }
    }
    \foreach \j in {0,1,2}{
      \draw[ordinary] (t0\j)--(t1\j)--(t2\j);
      \draw[ordinary] (t0\j) to[bend left=32] (t2\j);
    }
    \foreach \i in {0,1,2}{
      \draw[ordinary] (t\i0)--(t\i1)--(t\i2);
      \draw[ordinary] (t\i0) to[bend right=32] (t\i2);
    }
    \draw[edgeE] (t00)--node[below,font=\scriptsize,text=markedE] {$e$} (t10);
    \draw[edgeF] (t01)--node[left,font=\scriptsize,text=markedF] {$f$} (t02);
    \node[font=\tiny,below=1pt of t00] {$(0,0)$};
    \node[font=\tiny,below=1pt of t10] {$(1,0)$};
    \node[font=\tiny,left=1pt of t01] {$(0,1)$};
    \node[font=\tiny,left=1pt of t02] {$(0,2)$};
    \draw[<->,densely dashed,black!50] (2.35,-1.25)--node[right,font=\tiny] {$y\mapsto-y$} (2.35,1.25);
  \end{scope}

  \begin{scope}[xshift=10.75cm]
    \node[figlabel] at (-2.25,2) {(c)};
    \node[font=\scriptsize\bfseries] at (0,1.78) {Petersen};
    \foreach \i/\ang in {0/90,1/18,2/-54,3/-126,4/162}
      \node[circle,fill=black,inner sep=1.15pt] (p\i) at (\ang:1.15) {};
    \foreach \i/\ang in {0/90,1/18,2/-54,3/-126,4/162}
      \node[circle,fill=black,inner sep=1.15pt] (p\i i) at (\ang:.52) {};
    \draw[ordinary] (p1)--(p2)--(p3)--(p4)--(p0);
    \draw[ordinary] (p0i)--(p2i)--(p4i)--(p1i)--(p3i)--(p0i);
    \foreach \i in {0,1,2,4}
      \draw[ordinary] (p\i)--(p\i i);
    \draw[edgeE] (p0)--node[above right,font=\scriptsize,fill=white,
      inner sep=1pt,text=markedE] {$e$} (p1);
    \draw[edgeF] (p3)--node[below right,font=\scriptsize,fill=white,
      inner sep=1pt,text=markedF] {$f$} (p3i);
  \end{scope}
\end{tikzpicture}%
}

\newcommand{\KFourChainFigure}{%
\begin{tikzpicture}[x=1cm,y=1cm,>=Latex]
  \filldraw[fill=softGold!20,draw=softGold!75,rounded corners=5pt]
    (1.72,-.98) rectangle (4.68,.98);
  \foreach \i/\x in {1/0,2/2.15,3/4.3,4/6.45}{
    \node[vtx] (x\i) at (\x,.72) {$x_{\i}$};
    \node[vtx] (y\i) at (\x,-.72) {$y_{\i}$};
    \draw[ordinary] (x\i)--(y\i);
  }
  \foreach \i/\j in {1/2,2/3,3/4}{
    \draw[ordinary] (x\i)--(x\j) (y\i)--(y\j)
      (x\i)--(y\j) (y\i)--(x\j);
  }
  \draw[edgeE] (x1)--node[left,font=\scriptsize,text=markedE] {$e$} (y1);
  \draw[edgeF] (x4)--node[right,font=\scriptsize,text=markedF] {$f$} (y4);
  \node[font=\scriptsize\bfseries] at (3.2,1.23) {one $K_4$ block};
  \draw[->,black!55] (3.2,1.08)--(3.2,.92);
  \node[font=\scriptsize,align=center] at (2.15,-1.28) {shared terminal\\pair};
  \draw[->,black!55] (2.15,-1.02)--(2.15,-.84);
  \node[flowbox] at ($(y2)!0.5!(y3)+(0,-1.25)$)
    {$(a_k,b_k)^T\ \xrightarrow{\ M\ }\ (a_{k+1},b_{k+1})^T$};
\end{tikzpicture}%
}

\newcommand{\InjectionFigure}{%
\begin{tikzpicture}[x=1cm,y=1cm,>=Latex]
  \begin{scope}[xshift=0cm,yshift=0cm]
    \node[figlabel] at (-.9,1.15) {(a)};
    \node[font=\scriptsize\bfseries] at (2.1,1.0) {Adjacent edges};
    \node[blob] (A) at (0,0) {$[a]$};
    \node[blob] (B) at (2.1,0) {$[b]$};
    \node[blob] (C) at (4.2,0) {$[c]$};
    \draw[edgeE] (A)--node[above,font=\scriptsize,text=markedE] {$e$} (B);
    \draw[edgeF] (B)--node[above,font=\scriptsize,text=markedF] {$f$} (C);
    \draw[added] (A) to[bend right=34] (C);
    \node[font=\scriptsize] at (2.1,-1.0) {$b_t\longmapsto j_t$};
  \end{scope}

  \begin{scope}[xshift=7.25cm,yshift=0cm]
    \node[figlabel] at (-.9,1.15) {(b)};
    \node[font=\scriptsize\bfseries] at (2.1,1.0) {Join the outer blocks};
    \node[blob] (U) at (0,0) {$U$};
    \node[blob] (M) at (2.1,0) {$M$};
    \node[blob] (V) at (4.2,0) {$V$};
    \draw[edgeE] (U)--node[above,font=\scriptsize,text=markedE] {$e$} (M);
    \draw[edgeF] (M)--node[above,font=\scriptsize,text=markedF] {$f$} (V);
    \draw[added] (U) to[bend right=34] (V);
    \node[font=\scriptsize] at (2.1,-1.0) {$b_t\longmapsto j_t$};
  \end{scope}

  \begin{scope}[xshift=0cm,yshift=-4cm]
    \node[figlabel] at (-.9,1.3) {(c)};
    \node[font=\scriptsize\bfseries] at (2.1,1.15) {Use an outside component};
    \node[blob] (U2) at (0,.15) {$U$};
    \node[blob] (M2) at (2.1,.15) {$M$};
    \node[blob] (V2) at (4.2,.15) {$V$};
    \node[blob] (X) at (2.1,-1.05) {$X\ni x$};
    \draw[edgeE] (U2)--node[above,font=\scriptsize,text=markedE] {$e$} (M2);
    \draw[edgeF] (M2)--node[above,font=\scriptsize,text=markedF] {$f$} (V2);
    \draw[added] (X)--node[below left,font=\scriptsize,text=addedEdge] {$xu$} (U2);
    \draw[added] (X)--node[below right,font=\scriptsize,text=addedEdge] {$xv$} (V2);
    \node[font=\scriptsize] at (2.1,-1.9) {$b_t\longmapsto j_{t-1}$};
  \end{scope}

  \begin{scope}[xshift=7.25cm,yshift=-4cm]
    \node[figlabel] at (-.9,1.3) {(d)};
    \node[font=\scriptsize\bfseries] at (2.1,1.15) {Different marked components};
    \node[blob] (a2) at (0,.35) {$a$};
    \node[blob] (b2) at (4.2,.35) {$b$};
    \node[blob] (c2) at (0,-.85) {$c$};
    \node[blob] (d2) at (4.2,-.85) {$d$};
    \draw[edgeE] (a2)--node[above,font=\scriptsize,text=markedE] {$e$} (b2);
    \draw[edgeF] (c2)--node[below,font=\scriptsize,text=markedF] {$f$} (d2);
    \draw[added] (a2)--node[left,font=\scriptsize,text=addedEdge] {$ac$} (c2);
    \draw[added] (b2)--node[right,font=\scriptsize,text=addedEdge] {$bd$} (d2);
    \node[font=\scriptsize] at (2.1,-1.9) {$b_t\longmapsto j_{t-1}$};
  \end{scope}

  \node[font=\scriptsize,align=center] at (6.0,-6.25)
    {red: $e$\qquad blue: $f$\qquad green: edges added by the injection};
\end{tikzpicture}%
}

\newcommand{\CycleTwoForestFigure}{%
\begin{tikzpicture}[x=1cm,y=1cm,>=Latex]
  \begin{scope}[xshift=0cm]
    \node[font=\scriptsize\bfseries] at (0,1.55) {(a) A common cycle in $H+e+f$};
    \node[vtx] (a) at (-1.55,.75) {$a$};
    \node[vtx] (b) at (1.55,.75) {$b$};
    \node[vtx] (c) at (1.55,-.75) {$c$};
    \node[vtx] (d) at (-1.55,-.75) {$d$};
    \node[tinyvtx] (p1) at (2.0,0) {};
    \node[tinyvtx] (p2) at (-2.0,0) {};
    \draw[edgeE] (a)--node[above,font=\scriptsize,text=markedE] {$e$} (b);
    \draw[ordinary] (b)--(p1)--(c) node[midway,right,font=\tiny] {$P_1\subset H$};
    \draw[edgeF] (c)--node[below,font=\scriptsize,text=markedF] {$f$} (d);
    \draw[ordinary] (d)--(p2)--(a) node[midway,left,font=\tiny] {$P_2\subset H$};
  \end{scope}

  \draw[->,line width=.8pt] (2.5,0)--node[above,font=\scriptsize,align=center] {remove $e,f$\\and extend} (4.1,0);

  \begin{scope}[xshift=6.4cm]
    \node[font=\scriptsize\bfseries] at (0,1.55) {(b) A crosswise spanning two-forest};
    \filldraw[fill=softGold!25,draw=black!28] (-1.35,0) ellipse (.72cm and 1.22cm);
    \filldraw[fill=black!5,draw=black!28] (1.35,0) ellipse (.72cm and 1.22cm);
    \node[vtx] (a3) at (-1.35,.72) {$a$};
    \node[vtx] (d3) at (-1.35,-.72) {$d$};
    \node[vtx] (b3) at (1.35,.72) {$b$};
    \node[vtx] (c3) at (1.35,-.72) {$c$};
    \node[tinyvtx] (l) at (-1.35,0) {};
    \node[tinyvtx] (r) at (1.35,0) {};
    \draw[ordinary] (a3)--(l)--(d3);
    \draw[ordinary] (b3)--(r)--(c3);
    \draw[edgeE,ghost] (a3)--node[above,font=\scriptsize,text=markedE] {$e$} (b3);
    \draw[edgeF,ghost] (d3)--node[below,font=\scriptsize,text=markedF] {$f$} (c3);
    \draw[ordinary] (l)--++(-.65,.35) (l)--++(-.6,-.4) (r)--++(.65,.35) (r)--++(.6,-.4);
    \node[font=\scriptsize] at (0,-1.55) {counted by $\tau_2^\times$};
  \end{scope}

  \begin{scope}[xshift=11.9cm]
    \node[font=\scriptsize\bfseries] at (0,1.55) {(c) No common cycle};
    \filldraw[fill=markedE!6,draw=markedE!35] (-.85,0) circle (.95cm);
    \filldraw[fill=markedF!6,draw=markedF!35] (.85,0) circle (.95cm);
    \node[tinyvtx] (le1) at (-1.25,.35) {};
    \node[tinyvtx] (le2) at (-.55,.35) {};
    \node[tinyvtx] (le3) at (-.9,-.42) {};
    \node[tinyvtx] (rf1) at (.55,.35) {};
    \node[tinyvtx] (rf2) at (1.25,.35) {};
    \node[tinyvtx] (rf3) at (.9,-.42) {};
    \draw[ordinary] (le1)--(le2)--(le3)--cycle;
    \draw[ordinary] (rf1)--(rf2)--(rf3)--cycle;
    \draw[edgeE] (le1)--node[above,font=\tiny,text=markedE] {$e$} (le2);
    \draw[edgeF] (rf1)--node[above,font=\tiny,text=markedF] {$f$} (rf2);
    \node[font=\scriptsize,text=markedE] at (-.85,-1.2) {$M_e$};
    \node[font=\scriptsize,text=markedF] at (.85,-1.2) {$M_f$};
    \node[font=\scriptsize] at (0,-1.55) {$N_e\ \perp\ N_f$};
  \end{scope}
\end{tikzpicture}%
}

\title{Negative Correlation for Forests and the \\ $q<1$ Random-Cluster Model}
\author{
R. Altar \c{C}i\c{c}eksiz\thanks{Ume\aa\ University, Sweden. \href{mailto:altar.ciceksiz@umu.se}{altar.ciceksiz@umu.se}}
\and
Mohan Ravichandran\thanks{Idanta.ai. \href{mailto:mohan.ravichandran@gmail.com}{mohan.ravichandran@gmail.com}}
}
\date{}

\begin{document}
\maketitle

\begin{abstract}
We study negative edge correlation for two well-known models in statistical physics, the arboreal gas and the $q < 1$ random-cluster model.
%using a new, simple combinatorial observation: 
We show that after the edges $e$ and $f$ are removed, their Rayleigh difference is a crossing contribution minus the covariance of two endpoint-connectivity events.  For the arboreal gas on an arbitrary finite graph, we analyze the two leading coefficients of this difference.  The first is the classical transfer-current square. If it vanishes, the second has an electrical sum-of-squares formula.  This proves negative edge correlation at sufficiently large fugacity whenever one of the two leading coefficients of the Rayleigh difference is nonzero, and characterizes simultaneous equality by an equipotential decomposition.

For the random-cluster model on the complete graph $K_n$, we prove negative edge correlation throughout $0<q<1$ when all edge weights are at least $2$. In the uniform case, we can prove something stronger: negative correlation also holds when all the weights are equal and at least $1$.  In particular, distinct edges in a connected spanning subgraph of $K_n$, weighted by a fixed fugacity raised to its number of edges, are negatively correlated for every $n$. We also explore what seems to be a general phenomenon: positive correlation of connectivity events.  For the arboreal gas, we prove positive correlation for connectivity events at high fugacities; on lattices, the conjectured inequality would make the two-point function supermultiplicative and produce a convex inverse correlation length.  Exact coefficientwise computations support the graphical and matroidal conjectures, including on the Seymour--Welsh matroid $\mathcal S_8$, where ordinary edge-negative correlation fails.

The results in this paper were derived by the authors without the use of Large Language models. The authors did benefit from using GPT-5 Pro for generating code to test out hypotheses as well as for simplifying the arguments. 
%This paper benefited from the use of Large Language models. The result on edge negative correlation for the $q < 1$ random cluster model on $K_n$ was proved by us without the use of large language models, though we used GPT-5 pro for simplifying the proof. The other main result, negative correlation for large fugacity forests on general graphs was also human generated, but the authors also used GPT 5.1 for several routine but intricate symbolic calculations. The authors checked and verified all results and take responsibility for any errors or omissions.  
\end{abstract}



%\tableofcontents

\section{Introduction}\label{sec:introduction}

Correlation inequalities for random graphs lie at the intersection of combinatorics, probability and statistical physics. Over the last 50 years, researchers have developed a robust theory of positive correlation and association and applied it in the study of models such as the ferromagnetic Ising model and the random-cluster model with $q\geq 1$. This theory has benefited from a salubrious circumstance: The FKG inequality provides a local-to-global criterion for positive correlation/association and this makes it easy to prove that many models satisfy global positive association \cite{grimmett-1995, grimmett-2006}.

On the other hand, the corresponding theory for negative association is far less developed even though the need for it has been explicitly stated in many papers \cite{pemantle2000towards,borcea-branden-liggett-2009}. There is no general local-to-global principle for negative association, and even pairwise negative correlation questions on finite graphs are poorly understood; see \cite{pemantle2000towards} for an account of this matter. 

In this paper, we study two well-known models, \emph{the forest measure} and \emph{the random-cluster measure} with $q\leq 1$, and their respective negative dependence properties. Let $G=(V,E)$ be a graph with vertex set $V$ and edge set $E$. For a probability measure $\mu$ on $\{0,1\}^{E}$, we write $X_e$ for the indicator random variable of the edge $e\in E$. The measure $\mu$ is called \emph{negatively associated} if 
\[\Exp_\mu[FG]\leq \Exp_\mu[F]\Exp_\mu[G]\]
whenever $F$ and $G$ are increasing functions depending on disjoint sets of coordinates. We call the measure \emph{negatively edge correlated} if the weaker pairwise condition
\[\Exp_\mu[X_eX_f]\leq \Exp_\mu[X_e]\Exp_\mu[X_f]\]
holds for all $e\neq f\in E$. Note that this can be written as $\Cov_\mu(X_e,X_f)\leq 0$. Although weaker than negative association, negative edge correlation itself contains delicate enumerative information and is in general very hard to show. 

The most important example of a measure exhibiting negative correlation and association is the spanning tree measure.

\begin{thm}[Kirchhoff, Feder-Mihail]
Let $G = (V, E)$ be a finite graph and let $\mu$ be the probability measure on $\{0, 1\}^E$ that is uniform on the spanning trees of $G$. Then $\mu$ is not only negatively edge correlated but also negatively associated.  
\end{thm}

Negative edge correlation is due to Kirchhoff and also holds for weighted spanning tree measures (where each edge has a nonnegative weight and the weight of a spanning tree is the product of the edge weights over the edges in the tree). Negative association was proved by Feder and Mihail in \cite{feder1992balanced}.

One of the strongest techniques developed to show negative correlation is the theory of strongly Rayleigh measures of Borcea, Bränden and Liggett \cite{borcea-branden-liggett-2009}. It includes the spanning tree measure which was originally shown to be negatively edge correlated by Kirchhoff \cite{kir} and which was elaborated by Brooks, Smith, Stone and Tutte \cite{bsst1940,DoyleSnell1984}. Its modern probabilistic form is the transfer-current theorem of Burton and Pemantle \cite{burton-pemantle-1993}.  

The first model we will be concerned with is the generalization of the uniform spanning tree measure called the multivariate arboreal gas, or the multivariate forest measure. We let $\mathcal{F}(G)$ denote the set of spanning forests of $G$ and define the partition function for the spanning forest measure with weights $\mathbf{z}: E\to \mathbb{R}^{\geq 0}$, $\{z_g\}_{g\in E(G)}$ as
\[Z_G^{\mathcal{F}}(\mathbf{z})=\sum_{F\in \mathcal{F}(G)}\prod_{g\in F}z_g.\]

When $z_g=\lambda$ for all $g\in E(G)$, we will sometimes call the common weight the fugacity, following terminology from statistical physics. In this case, the partition function becomes 
\[Z_G^{\mathcal{F}}(\lambda)=\sum_{F\in \mathcal{F}(G)}\prod_{g\in F}\lambda=\sum_{k=1}^{|V(G)|}\tau_k(G)\lambda^{|V(G)|-k},\]
where $\tau_k(G)$ counts the number of unrooted forests of $G$ with $k$ components. Thus, a forest $F\in \mathcal{F}(G)$ has probability \[\Prob(F)=\dfrac{\lambda^{|F|}}{Z_G^{\mathcal{F}}(\lambda)}.\]
We note that $\tau_1(G):=\tau(G)$ is the number of spanning trees and when $\lambda=1$, $Z_G^{\mathcal{F}}(1):=Z_G^{\mathcal{F}}$ is just the partition function (normalizer) of the uniform spanning forest measure, and the probability for a forest $F$ is just \[\Prob(F)=\dfrac{1}{Z_G^{\mathcal{F}}}.\]
Moreover, on a connected graph, the arboreal gas measure converges to the uniform spanning tree measure as $\lambda \to \infty$ \cite{grimmett-2006}.

Let us note here the connection between the partition function and negative edge correlation. Let $Z_G(\mathbf{z})$ be a multi-affine partition function on edge subsets, not necessarily on forests, of $G$. For distinct edges $e,f$, we define the \emph{Rayleigh difference} as
\[\mathcal{R}_{ef}(Z_G):=(\partial_{z_e} Z_G)(\partial_{z_f}Z_G)-Z_G(\partial_{z_ez_f}Z_G).\]
Henceforth, when there is no confusion, we will write $\partial_{z_e}$ as $\partial_e$ as a shorthand. All Rayleigh differences are formed from the multivariate partition function before any specialization of the edge weights. We write $\Rayleigh_{ef}(Z_G)(\lambda\mathbf1)$ for evaluation at $z_g=\lambda$ for every edge $g$. Under the probability measure defined using the monomials of $Z_G$, which are edge subsets of $G$, we have
\[\Cov(X_e,X_f)=-\frac{z_ez_f}{Z_G^2}\Rayleigh_{ef}(Z_G).\]
Thus, negative edge correlation under this measure corresponds to $\Rayleigh_{ef}(Z_G)\geq 0$. 

In this terminology, the (in)famous conjecture for the arboreal gas measure (and so the forest measure), often attributed to Kahn, can be stated as follows.

\begin{conj}[Kahn \cite{kahn-2000}; see also \cite{grimmett-winkler-2004}]\label{conj:kahn}
For every finite graph $G$, every $\lambda>0$, and all distinct edges $e$ and $f$,
\[
\Rayleigh_{ef}(\ForestZ_G)(\lambda\mathbf1)\geq0.
\]
\end{conj}
The conjecture is known for small graphs and for series-parallel graphs \cite{grimmett-winkler-2004,semple2008negative,wagner2008negatively}; it remains open even for complete graphs in full generality. Even the uniform forest measure case when $\lambda=1$ is not known in full generality and is called the uniform forest measure conjecture. Stark proved that the uniform forest measures are negatively edge correlated on $K_n$ for all sufficiently large $K_n$ \cite{stark2011edge}.\footnote{Stark's proof carries over verbatim when every edge has the same fixed weight $\lambda>0$.  Consequently, for each fixed $\lambda$, there exists $n_0(\lambda)$ such that the arboreal gas on $K_n$ is negatively edge correlated for every $n\geq n_0(\lambda)$.}  Huang obtained large-fugacity results for adjacent edges in arbitrary graphs, and for arbitrary pairs in $K_n$ subject to additional size assumptions \cite{huang2024negative}. Lorentzian polynomials and completely log-concave polynomials give a sharp approximate inequality in all regimes and for all graphs, with a relaxed factor of $2-\frac{2}{|E(G)|}$ in place of the desired factor of 1 \cite{anari2019log,branden2020lorentzian,huh2021correlation}. We have
\[Z_G\partial_{ef}Z_G\leq 2\left(1-\frac{1}{|E(G)|}\right)\partial_eZ_G \partial_f Z_G.\]

The approach using Lorentzian polynomials has an obstruction that appears to be genuinely graphical. As shown by Seymour and Welsh, the matroid $\mathcal{S}_8$ does not exhibit negative edge correlation \cite{Seymour_Welsh_1975}. However, in the matroid setting, Feder and Mihail showed that balanced matroids exhibit negative edge correlation \cite{feder1992balanced}. Thus, we observe that neither abstract log concavity nor Lorentzian polynomials, nor the passage to matroids as a generalization by itself explains Kahn's conjecture. 

Wagner \cite{wagner2008negatively} conjectured the following strengthening of Kahn's conjecture. 
\begin{conj}
\label{conj:sumsofsquares}
Let $G=(V,E)$ be a finite graph, and let $e,f\in E$ be distinct edges. The Rayleigh difference $\Rayleigh_{ef}(\ForestZ_G)$ can be expressed as a finite sum of monomials in the edge weights multiplied by squares of polynomials in those weights. See Wagner~\cite{wagner2008negatively} for the more precise combinatorial formulation. 
\end{conj}

If true, Conjecture \ref{conj:sumsofsquares} implies Conjecture \ref{conj:kahn} since the sums of squares form guarantees positivity immediately. Erickson proved that the sums of squares form holds for the arboreal gas when the underlying graph is a series-parallel graph \cite{Erickson2012SOS} which can be considered as initial evidence for the conjecture. Our first main theorem follows this philosophy. We prove these when $\lambda$ is sufficiently large for all graphs.

%For equal fugacity, the deleted-graph expression differs from the Rayleigh difference by a factor of $\lambda^2$:
%\[
%\lambda^2\Rayleigh_{ef}(\ForestZ_G)(\lambda\mathbf1)
%=\ForestZ_{G-e}(\lambda)\ForestZ_{G-f}(\lambda)
%-\ForestZ_G(\lambda)\ForestZ_{G-e-f}(\lambda).
%\]
%This follows from multi-affinity; the general identity with factor $z_ez_f$ is proved in Lemma~\ref{lem:deletion-identity}. 

We have the following theorem for the second-highest term when the top vanishes. The electrical terms and the pseudoinverse used below will be defined in the appropriate section. We first set up some notation and state the transfer-current theorem. 

\begin{thm}[Burton-Pemantle]
\label{thm:transfer-current}

Let $G=(V,E)$ be a connected graph on $n$ vertices, and let $e,f$ be distinct non-bridge edges.  We orient the edges arbitrarily, and put $b_h=\1_{h^+}-\1_{h^-}$. Let $L$ be the graph Laplacian, and $L^\dagger$ be the Moore-Penrose pseudoinverse. We set
\[
v_h=L^\dagger b_h,
\qquad
u_h(g)=b_g^TL^\dagger b_h.
\]
We can write
\[
\Rayleigh_{ef}(\ForestZ_G)(\lambda\mathbf1)
=H_0\lambda^{2n-4}+H_1\lambda^{2n-5}+O(\lambda^{2n-6}).
\]
Then, we have
\[
H_0=\tau(G)^2u_e(f)^2.
\]
\end{thm}

We now state our first major theorem.

\begin{thm}
\label{thm:leading-coefficients}
Use the same notation as in the statement of Theorem (\ref{thm:transfer-current}).

If $H_0=0$, we choose additive constants so that the potentials $\widetilde v_e,\widetilde v_f$ satisfy
\[
\widetilde v_e(f^-)=\widetilde v_e(f^+)=0,
\qquad
\widetilde v_f(e^-)=\widetilde v_f(e^+)=0.
\]
Then
\begin{equation*}\label{eq:main-sos}
H_1=\tau(G)^2\sum_{g=xy\in E(G)}
\bigl(\widetilde v_e(x)\widetilde v_f(y)
-\widetilde v_e(y)\widetilde v_f(x)\bigr)^2.
\end{equation*}
In particular, $H_0,H_1\geq0$ in the stated order.
\end{thm}

When the edges $e$ and $f$ are adjacent, the situation is transparent. 

\begin{rmk}
When the edges $e$ and $f$ are adjacent, say $e = ab$ and $f = bc$, it is easy to see that $H_0$ is zero exactly when $b$ is a cut vertex. When this happens, we have that $X_e$ and $X_f$ are actually independent. We leave the (simple) proof of this fact to the interested reader. 
\end{rmk}

We now characterize the case when both the top two coefficients are zero.

\begin{corollary}\label{cor:large-fugacity}
Let $G=(V,E)$ be a graph, and let $e,f$ be distinct non-bridge edges. Under the hypothesis of Theorem~\ref{thm:leading-coefficients}, there exists a constant $C_G$ such that for all $\lambda>C_G$, $X_e$ and $X_f$ are negatively correlated, or we have $\widetilde{v}_e(x)\widetilde{v}_f(y)=\widetilde{v}_e(y)\widetilde{v}_f(x)$ for all edges $xy\in E$. This implies that every pair of distinct edges of $K_n$, and the complete-join, toroidal, and Petersen examples in Section~\ref{sec:vanishing} exhibit negative correlation for large enough $\lambda$.
\end{corollary}

The equality case in the above theorem has a geometric meaning: the two unit-current potentials are separated by a common subset that is equipotential for both currents.  In Section~\ref{sec:vanishing} we explain why symmetry can force the spanning-tree term to vanish without forcing the entire Rayleigh difference to be zero.  In particular, chains of $K_4$ blocks exhibit arbitrarily long strings of vanishing leading coefficients even though every surviving coefficient is positive.

Next, we describe the random-cluster model which is the second model we study in this paper. We let $\kappa(G)$ denote the number of components of a graph. Then, the partition function for the random-cluster model with parameter $q>0$ is

\[Z_G^{\mathrm{RC}}(q,\mathbf{z}):=\sum_{A\subseteq E(G)}q^{\kappa(A)}\prod_{g\in A}z_g.\]

Under this model, a subset of edges $A\subseteq E(G)$, which we call a \emph{configuration}, has probability $\frac{q^{\kappa(A)}\prod_{g\in A}z_g}{Z_G^{\mathrm{RC}}(q,\mathbf{z})}$. Fortuin and Kasteleyn introduced the model to unify the study of bond percolation and Ising/Potts models \cite{fortuin-kasteleyn-1972}. We briefly mention basic properties of the model here; see \cite{grimmett-2006} for a comprehensive treatment of the model. 

At $q=1$, the connectivity parameter in the partition function vanishes, and the model becomes the independent bond percolation model with edge weights $\mathbf{z}$, and at $q=2$, the model is coupled to the Ising model via a simple relation. 

We have several limits when $q\to 0$ in distinct regimes with respect to edge weights $\mathbf{z}$, and each limit recovers known models. If $z_g=q\lambda_g$, the smallest powers of $q$ select forests and the weak limit is the multivariate arboreal gas.  If the edge weights are instead held fixed, the smallest component count is selected; on a connected host graph the limit is the weighted uniform measure on connected spanning subgraphs.  Symbolically,
\[
\RCZ_G(q,q\boldsymbol\lambda)
\xrightarrow[q\to0]{\text{after normalization}}
\ForestZ_G(\boldsymbol\lambda),
\qquad
\RCZ_G(q,\mathbf z)
\xrightarrow[q\to0]{\text{after normalization}}
\sum_{\substack{A\subseteq E(G)\\(V,A)\ \mathrm{connected}}}\prod_{g\in A}z_g.
\]

When $q\geq 1$, using the FKG inequality, one can show that the model is positively associated \cite{grimmett-2006}. However, the FKG inequality stops working below $q=1$, and the model changes correlation properties at $q=1$. Grimmett and Winkler conjectured that for $0<q<1$, the model is instead negatively associated. Even negative edge correlation remains open for arbitrary host graphs and weights for $0<q<1$ \cite{grimmett-winkler-2004}.

\begin{conj}[Grimmett--Winkler]\label{conj:grimmett-winkler}
Let $0<q<1$ and let all edge weights be nonnegative.  Then, for every finite graph $G$ and distinct edges $e,f$,
\[
\Rayleigh_{ef}(\RCZ_G)(q,\mathbf z)\geq0.
\]
\end{conj}

Our second main result proves this pairwise assertion on complete graphs when all edges have a common weight at least $1$, throughout the whole interval $0<q<1$.

\begin{thm}\label{thm:complete-graph}
Let $0<q<1$ and $\lambda\geq1$.  For the random-cluster model on $K_n$ with all edge weights equal to $\lambda$, every two distinct edge indicators are negatively correlated.  Equivalently,
\[
\Rayleigh_{ef}(\RCZ_{K_n})(q,\lambda\1)\geq0
\qquad(e\neq f).
\]
More strongly, after normalizing the two marked weights, every coefficient of the power-series expansion of $\Rayleigh_{ef}(\RCZ_{K_n})/(1-q)^2$ at $q=0$ is nonnegative as a function of $\lambda\geq1$.
\end{thm}

\begin{corollary}\label{cor:connected-subgraphs}
For every $n$ and $\lambda\geq1$, distinct edges in a connected spanning subgraph $A$ of $K_n$, sampled with probability proportional to $\lambda^{|A|}$, are negatively correlated.  In particular this holds for the uniform measure.
\end{corollary}

Tang and Zhang recently proved Corollary~\ref{cor:connected-subgraphs} for all sufficiently large $n$ \cite{tang-zhang-2026}. Here, there is no large-$n$ restriction.  Theorem~\ref{thm:complete-graph} is also stronger in a different direction: the connected-subgraph measure is its $q\to0$ specialization with fixed edge weights, whereas the theorem covers the entire interval $0<q<1$.

We note that in Theorem \ref{thm:complete-graph}, the restriction to complete graphs and a common weight is not just cosmetic and is actually substantive.  The proof uses enough symmetry to compare configuration classes by explicit injections and component-count estimates.  It does not establish the Grimmett--Winkler conjecture for arbitrary graphs or arbitrary weights, nor does it assert negative association.  Proposition~\ref{prop:weighted-cross-edge} records a nonuniform weighted range in which the injection survives.

\section*{Statement of AI use}
This paper benefited from the use of Large Language models. The result on edge negative correlation for the $q < 1$ random cluster model on $K_n$ was proved by us without the use of large language models, though we used GPT-5 Pro for simplifying the proof. The other main result, negative correlation for large fugacity forests on general graphs was also human generated, but the authors also used GPT-5 pro for several routine but intricate symbolic calculations. The authors checked and verified all results and take responsibility for any errors or omissions. 

\section{Overview of the Proofs}

Even though the two models and the proofs of the main theorems differ in the technical details, in both cases we will employ the following philosophical framework. For the Rayleigh difference in edges $e$ and $f$, instead of working directly with the graph, we delete the two edges and work in the graphs with these edges \emph{deleted}. We first prove Lemma \ref{lem:deletion-identity} to reduce the Rayleigh difference to four distinct components in deleted graphs $G-e$, $G-f$ and $G-\{e,f\}$, and $G$ itself. 

For the arboreal gas model, expanding the forest polynomials at large fugacity we arrive at two terms. The top term is the term coming from spanning trees and the second term is the term coming from $2$-component forests. Recall that we use $\tau_k(G)$ to denote the number of spanning forests of $G$ with exactly $k$ components. The top term is known via the transfer-current theorem to be expressed as sums of squares, and hence is nonnegative. If this does not vanish, we immediately have negative correlation for large $\lambda$. Hence, we work with the extra information that the top term corresponding to the spanning tree Rayleigh difference is $0$. Under this assumption, the next coefficient reduces to a deleted difference of the function \[Q(G)=\frac{\tau_2(G)}{\tau(G)}.\]

We then employ a formula by Richman, Shokrieh and Wu for the quantity $Q$ \cite{richman2023counting} using potential differences, combined with rank-one deletions to prove that the quantity reduces to a sums of squares form shown in \ref{thm:leading-coefficients}. This also makes the equality case accessible. That is, we show when the second-highest coefficient in the Rayleigh difference is exactly $0$. 

For the random-cluster measure on the complete graph, our approach is similar in spirit. We first condition on the edges being in the graph. Then, due to the nature of the random-cluster model, we analyze the expression depending on the edges being a bridge, a non-bridge or a joint bridge. Let $e=ab$, $f=cd$, and $H=G-e-f$. We denote by $N_e=\{a\leftrightarrow b\}$ and $N_f=\{c\leftrightarrow d\}$ the connectivity events of the edges' endpoints in $H$. Let $J_{ef}$ be the event that $e,f$ are not connected, but the two absent edges $e,f$ would join the same two components; in other words, they are joint bridges. We prove that
\[\dfrac{q^2\Rayleigh_{ef}(\RCZ_G)}{(1-q)(\RCZ_H)^2}=\Prob_H(J_{ef})-(1-q)\Cov_H(\mathbf{1}_{N_e},\mathbf{1}_{N_f}).\]

Dividing by $(1-q)^2$, the Rayleigh difference becomes a power series whose coefficients are sums and differences of corresponding configuration counts. We use injections to prove that these are positive everywhere. 

\subsection*{Organization of the Paper}
In Section \ref{sec: preliminarylemmas}, we prove the lemmas which help us carry the models into the deleted graph framework. We do this for general multi-affine partition functions, so the results work for both the arboreal gas measure and the random-cluster measure. In Section \ref{sec: arboreal-gas}, we start by introducing the electrical notation. We then introduce the main quantity, the two-forest ratio $Q(G)$, and use Richman-Shokrieh-Wu's formula along with the lemmas we establish to prove that when $\lambda$ is sufficiently large, we have negative edge correlation in the arboreal gas measure. We also characterize when the second-highest coefficient could be equal to zero. In Section \ref{sec:complete-proof}, we go on to prove the Rayleigh difference result for the random-cluster measure with $0<q<1$ in complete graphs. This is done by several lemmas establishing injective and analytical bounds on configurations containing an edge as a bridge and as a non-bridge.

\section{Preliminary Lemmas on Deleted Graphs}
\label{sec: preliminarylemmas}

For any multi-affine partition function $Z_G(\mathbf{z})$, the Rayleigh difference has the following deleted-graph representation.

\begin{lemma}
\label{lem:deletion-identity}
For distinct edges $e$ and $f$, 
\[Z_{G-e}Z_{G-f}-Z_GZ_{G-e-f}=z_ez_f\Rayleigh_{ef}(Z_G).\]
\end{lemma}

\begin{proof}
We can separate the monomials occurring in the partition function $Z_G$ with respect to their containment of $z_e,z_f$, both or neither. We write $Z_G=A+z_eB+z_fC+z_ez_fD$. Note that $A$ corresponds naturally to the partition function of $G-e-f$, and $A+z_fC$ and $A+z_eB$ correspond to the partition functions of $G-e$ and $G-f$ respectively. With these, both sides reduce to $z_ez_f(BC-AD)$.
\end{proof}

We fix $e=ab$, $f=cd$, and let $H=G-e-f$. Let $C_A(x)$ denote the component containing $x$ in the configuration $A\subseteq E(H)$ where $x\in V(G)$. We define formally
\[N_e=\{C_A(a)=C_A(b)\}, \;\;\;\;\; N_f=\{C_A(c)=C_A(d)\},\]
that is, the event that $a,b\in V(G)$ are connected in $A$ and similarly for $c,d \in V(G)$. Under the event $N_e^c\cap N_f^c$ that neither pair is connected, we define the joint bridge event 
\[J_{ef}=N_e^c\cap N_f^c\cap\{\{C_A(a),C_A(b)\}=\{C_A(c),C_A(d)\}\}.\]
For a configuration $A\subseteq E(G)$ containing both marked edges, we use
$J_{ef}$ for the pullback of this event under $A\mapsto A\setminus\{e,f\}$.
Thus its component conditions are always tested after deleting $e,f$;
when both are present, neither is a bridge individually, but deleting both
increases the component count by one.

Figure \ref{fig:deleted-graph-events} illustrates the point of the joint bridge event which we shall return to in subsequent sections. 

\begin{figure}[htbp]
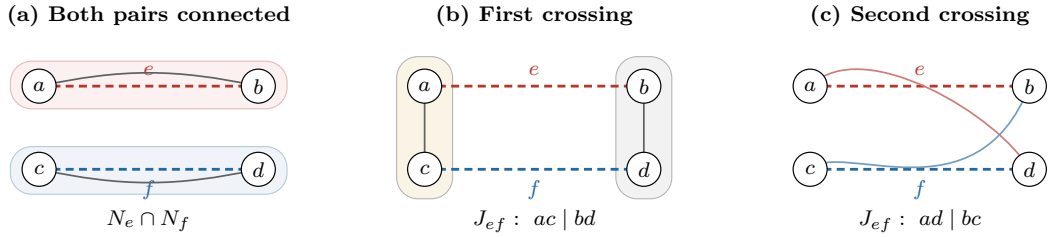

\centering
\DeletedGraphFigure
\caption{The deleted-graph events.  Dashed red and blue segments are the absent marked edges $e=ab$ and $f=cd$.  In (a), both endpoint pairs are already connected.  In (b) and (c), neither pair is connected, but adding both marked edges joins the same two components; these are the two realizations of the crossing event $J_{ef}$.}
\label{fig:deleted-graph-events}
\end{figure}

For an event $\mathcal{E}$, we write the probability of $\mathcal{E}$ using the partition function as follows. We define 
\[Z_H(\mathcal{E})=\sum_{A\subseteq E(H)}\mathbf{1}_\mathcal{E}(A)q^{\kappa(A)}\prod_{g\in A}z_g.\]
Hence, we have
\[\Prob_{Z_H}(\mathcal{E})=\dfrac{Z_H(\mathcal{E})}{Z_H}.\]

We also define the \emph{bridge correlation} using 
\[\Cov_{Z_H}(N_e,N_f)=\dfrac{\Gamma_H(a,b;c,d)}{Z_H^2},\]
where \[\Gamma_H(a,b;c,d):=Z_H(N_e\cap N_f)Z_H-Z_H(N_e)Z_H(N_f).\]

For either model $\star\in\{\mathcal F,\mathrm{RC}\}$ and any graph $K$ containing the indicated vertices, we use the endpoint-explicit notation
\[
\Gamma_K^\star(a,b;c,d):=
Z_K^\star(a\leftrightarrow b,\,c\leftrightarrow d)Z_K^\star
-Z_K^\star(a\leftrightarrow b)Z_K^\star(c\leftrightarrow d).
\]
Here $Z_K^\star(\mathcal E)$ is the partition sum restricted to the event
$\mathcal E$ in the indicated model. Weight arguments are suppressed unless
needed, as in $\Gamma_K^{\mathcal F}(a,b;c,d;\mathbf z)$ or its equal-weight
specialization $\Gamma_K^{\mathcal F}(a,b;c,d;\lambda)$.

Armed with all these, let us prove the connection between the bridge events and the Rayleigh difference for the random-cluster model.

\begin{prop}
\label{prop: deleted-graph-identity}
Let $\RCZ_G$ denote the partition function for the random-cluster model on $G$, and let $q\neq 1$. Fix $e,f\in E(G)$, and let $H=G-e-f$. We have
\[\dfrac{q^2\Rayleigh_{ef}(\RCZ_G)}{(1-q)(\RCZ_G)^2}=\Prob_{Z_H}(J_{ef})-(1-q)\Cov_{Z_H}(N_e,N_f).\]
\end{prop}

\begin{proof}
First, observe that in the notation of Lemma \ref{lem:deletion-identity}, we have $\Rayleigh_{ef}(Z)=BC-AD$, and thus the Rayleigh difference of a multi-affine partition function is independent of the variables $z_e,z_f$. We set $z_e=z_f=0$. Then the derivatives have the direct interpretation of adding the edges $e$ and $f$ to a configuration of $H$. 

Now, let us note what adding the edge $e$ implies. If $N_e$ fails, that is, if the endpoints of the edge $e$ are disconnected in the graph, adding $e$ decreases $\kappa$ by $1$, and thus multiplies the random-cluster weight of the configuration by $q^{-1}$. Otherwise, if $N_e$ holds, there is no change and the weight of the configuration stays the same. Thus, we have
\[\dfrac{q\partial_e \RCZ_G}{\RCZ_H}=1-(1-q)\Prob_H(N_e), \;\;\;\;\; \dfrac{q\partial_f \RCZ_G}{\RCZ_H}=1-(1-q)\Prob_H(N_f).\]

We do a similar analysis for adding the edges $e,f$ together. On $N_e\cap N_f$, adding $e,f$ merges no components. If one of $N_e$ or $N_f$ occurs, i.e. on $N_e\cap N_f^c$ and $N_e^c\cap N_f$, one component is merged and hence we multiply by $q^{-1}$, and on $N_e^c\cap N_f^c$, two components merge together, and hence the configuration is multiplied by $q^{-2}$. Partitioning according to the possible occurrences of $N_e,N_f$ yields
\[\dfrac{q^2\partial_{ef}\RCZ_G}{\RCZ_H}
=1-(1-q)\bigl(\Prob_H(N_e)+\Prob_H(N_f)\bigr)
+(1-q)^2\Prob_H(N_e\cap N_f)-(1-q)\Prob_H(J_{ef}).\]

Multiplying the corresponding partial derivatives to get the Rayleigh difference gives us 
\[\dfrac{q^2\Rayleigh_{ef}(\RCZ_G)}{(\RCZ_H)^2}
=(1-q)\Prob_H(J_{ef})-(1-q)^2\Cov_H(\1_{N_e},\1_{N_f}),\]
which is what we wanted.
\end{proof}

We use Proposition \ref{prop: deleted-graph-identity} to show a similar result for the arboreal gas measure. First, note that the arboreal gas measure is obtained from the random-cluster measure by setting $z_g=q\lambda_g$ and letting $q\to 0$. Indeed, for a spanning subgraph $A$ on $n$ vertices, $\kappa(A)+|A|=n$ precisely when $A$ is a forest, and is larger otherwise; hence in the limit $q\to 0$, the mass of the measure becomes concentrated on forests. Using this observation and Proposition \ref{prop: deleted-graph-identity}, we have the following corollary.

\begin{corollary}
\label{cor:forest-deleted-graph}
Let $\ForestZ_G(\mathbf{\lambda})$ be the partition function for the arboreal gas measure. Let $e,f\in E(G)$, and $H=G-e-f$. We have 
\[
\dfrac{\Rayleigh_{ef}(\ForestZ_G)}{(\ForestZ_H)^2}
=\Prob_H^{\mathcal{F}}(J_{ef})
-\Cov_H^{\mathcal{F}}(\1_{N_e},\1_{N_f}).\]
Equivalently, if $\mathcal A_e=N_e^c$ is the event that $e$ can be added to the sampled forest without making a cycle, and $\mathcal A_{ef}$ is the event that both $e$ and $f$ can be added simultaneously, then
\[
\dfrac{\Rayleigh_{ef}(\ForestZ_G)}{(\ForestZ_H)^2}
=\Prob_H^{\mathcal{F}}(\mathcal A_e)\Prob_H^{\mathcal{F}}(\mathcal A_f)
-\Prob_H^{\mathcal{F}}(\mathcal A_{ef}).\]
\end{corollary}

\begin{proof}
Under the substitution $z_g=q\lambda_g$, we have $\RCZ_H=q^{|V(H)}(\ForestZ_H+O(q))$, and $\Rayleigh_{ef}(\RCZ_G)=q^{2|V(H)|-2}(\Rayleigh_{ef}(\ForestZ_G)+O(q))$. Taking the limit $q\to 0$ yields the first claim on the Rayleigh difference for the arboreal gas measure. 

For the second claim, note that $\mathcal{A}_e\cap \mathcal{A}_f$ is the disjoint union of the events $\mathcal{A}_{ef}$ and $J_{ef}$. Either both can be added simultaneously or they can be added individually but both cannot be added which corresponds to being joint bridges, $J_{ef}$. Next, note that covariance is unchanged under complementing both events; hence we have 
\[\Prob_{H}^{\mathcal{F}}(\mathcal{A}_{ef})+\Prob_H^{\mathcal{F}}(J_{ef})=\Prob_H^{\mathcal{F}}(\mathcal{A}_e\cap \mathcal{A}_f),\]
from which we get the second desired equality.
\end{proof}

The equalities in Proposition \ref{prop: deleted-graph-identity} and Corollary \ref{cor:forest-deleted-graph} give us the following inequalities on connectivity events $N_e$ and $N_f$ which are equivalent to negative edge correlation in their respective models, 
\begin{align*}
    \Cov_H^{\mathrm{RC}}(\mathbf{1}_{N_e},\mathbf{1}_{N_f})&\leq \dfrac{\Prob_H^{\mathrm{RC}}(J_{ef})}{1-q} \;\;\;\;\; (0<q<1) \\
      \Cov_H^{\mathcal{F}}(\mathbf{1}_{N_e},\mathbf{1}_{N_f})&\leq \Prob_H^{\mathcal{F}}(J_{ef}). 
\end{align*}

We will also discuss the inequality  $\Cov_H(\mathbf{1}_{N_e},\mathbf{1}_{N_f})\geq 0$ in Section \ref{sec:connectivity-programme}
which is another conjecture we are interested in. 

We end this section by defining the bridge correlation (bridge Rayleigh difference) for the random-cluster measure and showing how it ties to the Rayleigh difference of the random-cluster partition function. Let $e,f \in E(G)$ be fixed and let $W_{ij}$ for $i,j \in \{0,1\}$ be the random-cluster partition function restricted to configurations in which the edge $e$ has state $i$ and the edge $f$ has state $j$, where $i=0$ indicates that the edge is not in the configuration and $i=1$ that the edge is in the configuration. We write $B_e$ and $B_f$ for the events that the edges are in the configuration and they are bridges respectively. In the expansion used in Lemma~\ref{lem:deletion-identity}, these satisfy $W_{00}=A$, $W_{10}=z_eB$, $W_{01}=z_fC$, and $W_{11}=z_ez_fD$. Figure \ref{fig:ordinary-joint-bridges} shows the difference between the events $B_e,B_f$ and $J_{ef}$ of being joint bridges. 

\begin{figure}[htbp]
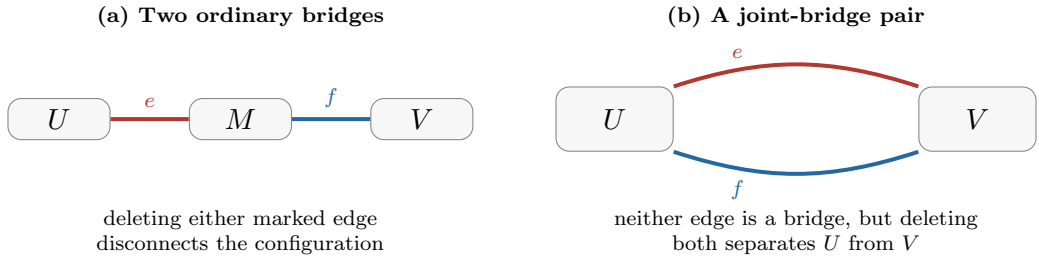

\centering
\BridgeTypesFigure
\caption{Ordinary and joint bridges, conditioned on both marked edges being present.  In (a), each marked edge is a bridge.  In (b), neither marked edge is a bridge, since the other still connects the two blocks; deleting both raises the component count by one.}
\label{fig:ordinary-joint-bridges}
\end{figure}

We define the bridge Rayleigh difference as
\[\Rayleigh_{ef}^b=\partial_{ef}\RCZ_G(B_e)\partial_{ef}\RCZ_G(B_f)-\partial_{ef}\RCZ_G(B_e\cap B_f)\partial_{ef}\RCZ_G.\]

We reserve $Z_{ef}:=\partial_{ef}\RCZ_G$ for the derivative, so that
$W_{11}=z_ez_fZ_{ef}$. For an event $\mathcal E$, define
\[
Z_{ef}(\mathcal E):=\sum_{\substack{A\subseteq E(G)\\e,f\in A}}
\mathbf1_{\mathcal E}(A)q^{\kappa(A)}\prod_{g\in A\setminus\{e,f\}}z_g.
\]
These are unnormalized restricted sums with the marked-edge weights removed;
the conditional probability of $\mathcal E$ given $e,f\in A$ is
$Z_{ef}(\mathcal E)/Z_{ef}$.

Hence, we are looking at the Rayleigh difference with respect to the edges being present and being bridges. We prove the following proposition which makes this definition worthwhile. 
\begin{prop}
\label{prop:bridge-identity}
For the random-cluster model $\RCZ_G$ with $q\neq 1$, we have
\[
\dfrac{\Rayleigh_{ef}(Z_G)}{(1-q)^2}
=\Rayleigh^b_{ef}
+\frac{Z_{ef}(J_{ef})Z_{ef}}{1-q}.
\]
\end{prop}

\begin{proof}
Let $\nu$ be the measure conditioned on both edges being present, i.e. it is the normalized measure on $W_{11}$. Set $p_e:=\nu(B_e), p_f:=\nu(B_f), p_{ef}:=\nu(B_e\cap B_f)$, and $p_J:=\nu(J_{ef})$. Deleting the edges $e,f$ one by one, we get 
\[\dfrac{z_fW_{10}}{W_{11}}=1-(1-q)p_f, \;\;\;\;\; \dfrac{z_eW_{01}}{W_{11}}=1-(1-q)p_e.\]
The idea here is just to note in which of the configurations the corresponding edge ($e$ or $f$) is a bridge and is not a bridge and to partition the functions accordingly noting that we have a factor of $q$ when it is a bridge. 

We also note the following identity for the connectivity function $\kappa$ when both $e$ and $f$ are present in the graph. For $A\in W_{11}$, we have
\[\kappa(A-e-f)-\kappa(A)=\mathbf{1}_{B_e}+\mathbf{1}_{B_f}+\mathbf{1}_{J_{ef}}.\]
Since the joint bridge event is disjoint from both being bridges, we get the following from above
\[\dfrac{z_ez_fW_{00}}{W_{11}}=1-(1-q)(p_e+p_f)+(1-q)^2p_{ef}-(1-q)p_J.\]

Now, if we multiply and arrange the quantities to get to the Rayleigh difference, we get 
\[z_ez_f\dfrac{W_{10}W_{01}-W_{11}W_{00}}{W^2_{11}}=(1-q)p_J+(1-q)^2(p_ep_f-p_{ef}).\]

Now note that 
\begin{align*}
Z_{ef}&=\dfrac{W_{11}}{z_ez_f}\\
\Rayleigh_{ef}&=\dfrac{W_{10}W_{01}-W_{11}W_{00}}{z_ez_f}\\
\Rayleigh^b_{ef}&=Z^2_{ef}(p_ep_f-p_{ef})\\
Z_{ef}(J_{ef})Z_{ef}&=p_JZ_{ef}^2.
\end{align*}
Using the above, and multiplying by $Z_{ef}^2$, and dividing by $(1-q)^2$, we get the desired equality relating the Rayleigh difference to the Rayleigh bridge difference. 
\end{proof}

\section{Arboreal gas measure at large-\texorpdfstring{$\lambda$}{lambda} fugacity}
\label{sec: arboreal-gas}
\subsection{Proof of Theorem \ref{thm:leading-coefficients}}
In this section, we prove Theorem \ref{thm:leading-coefficients}. Before delving into the details of the proof, we introduce some electrical and linear algebraic notation in order to familiarize the reader with this notation. 

Let $G=(V,E)$ be a graph. We can give $G$ an arbitrary orientation by directing each edge $h$. We pick the endpoints and arbitrarily assign $h^+,h^-$ to the endpoints. The edge will be directed from $h^-$ to $h^+$. We write $b_h=\mathbf{1}_{h^+}-\mathbf{1}_{h^-}$, where $\mathbf{1}_{v}\in \mathbb{R}^{|V|}$ with $\mathbf{1}_v(u)=\delta_{uv}$.\footnote{Here $\delta_{uv}$ is the usual Kronecker delta symbol: $\delta_{uv}=0$ when $u\neq v$ and $\delta_{vv}=1$.} 

Suppose a unit current is injected at $h^+$ and removed from $h^-$, i.e. view the graph as having $h^+$ as its source and $h^-$ as its sink. For any vertex $u\in V$, we denote the potential induced by this current as $v_h(u)$, and for a given edge $g=g^-g^+$, we denote the potential difference across its endpoints as $u_h(g):=v_h(g^+)-v_h(g^-)$. We reserve $v_h=L^\dagger b_h$ for the mean-zero potential and write $\widetilde v_h=v_h+c_h\mathbf1$ for a shifted representative, with its grounding specified locally. We write $s_h(g):=\widetilde v_h(g^+)+\widetilde v_h(g^-)$ for the endpoint sum of the representative in use; the mean-zero choice corresponds to $c_h=0$. We let $r_G(h)$ be the effective resistance along an edge, and note that $r_G(h)=u_h(h)$, and let $r_G(x,y)$ be the effective resistance between any two vertices $x$ and $y$. We also note that we have $u_h(g)=u_g(h)$ which is the transfer current from $h$ to $g$. We also define the effective conductance as the inverse of the resistance, $c^{eff}(h):=\frac{1}{r_G(h)}$.

Under unit conductances, we define the graph Laplacian using the vectors as \[L_G:=L=\sum_{h\in E}b_hb_h^T.\]
We will use the Moore-Penrose pseudoinverse of the Laplacian matrix to write the electrical parameters in terms of the Laplacian, which we denote as $L^{\dagger}:=L_G^{\dagger}$. For completeness, we define it here. Let $L\in \mathbb{C}^{n\times m}$. Then $L^{\dagger}\in \mathbb{C}^{m\times n}$ is the unique matrix satisfying the following three conditions \cite{cambridgeGeneralizedInverse}: 
\begin{enumerate}
    \item $LL^{\dagger}$ doesn't have to be the identity matrix but it maps all the column vectors of $L$ to themselves, i.e. $LL^{\dagger}L=L$.
    \item $L^{\dagger}$ acts like a weak inverse in the sense that we have \[L^{\dagger}LL^{\dagger}=L^{\dagger}.\]
    \item $LL^{\dagger}$ and $L^{\dagger}L$ are both Hermitian matrices, \[(LL^{\dagger})^{*}=LL^{\dagger}, \;\;\;\; (L^{\dagger}L)^*=L^{\dagger}L.\]
\end{enumerate}

With the pseudoinverse, we have 
\[v_h=L^{\dagger}b_h\]
for the voltage created by injecting one unit of current at $h^+$ and removing it at $h^-$. This function is harmonic except at the two endpoint vertices and is defined uniquely up to additive constants, and the pseudoinverse chooses the mean-zero constant. The potential difference we defined above can be written as \[u_h(g)=b_g^Tv_h=b_g^TL^{\dagger}b_h.\]

We also note that reversing the orientation of an edge changes the signs but it does not change the value, and hence is invariant under taking squares. 

We recall what Kirchhoff's matrix-tree theorem and the transfer-current theorem say in terms of the electrical parameters we defined for completeness. The main reason we do this calculation is that this becomes the highest term in our forest calculation as well. If $T$ is a uniform spanning tree under the measure $\mathcal{T}$, and $e \in E(G)$ then we have
\[\Prob_{\mathcal{T}}(e\in T)=r_G(e).\]
For the covariance of two edges $e,f$ we have
\[\Cov_{\mathcal{T}}(\mathbf{1}_{e\in T},\mathbf{1}_{f\in T})=-u_e(f)^2.\]
Recall that $\tau_1=\tau$ denoted the number of spanning trees. The above is then equivalent to
$\tau(G-e)=\tau(G)(1-r_G(e))$.

We can do a similar analysis for the graph $G-e-f$ obtained by deleting two edges to deduce that in spanning trees, we get a square term. 

In the usual matrix-tree theorem, we delete one row and one column of the Laplacian matrix; let us denote this by $\hat{L}$, and let us denote the deleted version of the incidence vector as $\hat{b}_h$. Then the matrix-tree theorem states that $\tau(G)=\operatorname{det}(\hat{L})$. Using the matrix determinant lemma, we get 
\[\tau(G-e-f)=\operatorname{det}(\hat{L}_{G-e-f})=\operatorname{det}(\hat{L})\operatorname{det}\left(I-\begin{pmatrix}
\hat{b}_e^T \\
\hat{b}_f^T
\end{pmatrix}\hat{L}^{-1}\begin{pmatrix}
\hat{b}_e & \hat{b}_f    
\end{pmatrix}\right).\]

The entries of the $2\times 2$ matrix on the right are $r_G(e),r_G(f)$ and $u_e(f)=u_f(e)$. Thus, expanding the determinant, and noting that $\operatorname{det}(\hat{L})=\tau(G)$, we get
\[\tau(G-e-f)=\tau(G)((1-r_G(e))(1-r_G(f))-u_e(f)^2),\]
where $r_G(e)=b_e^TL^\dagger b_e$. This turns out to be the same as applying $\hat{L}^{-1}$ in this calculation, because we are mainly interested in the potential differences and the additive constant does not matter in taking differences. 

With these at hand, let us calculate the Rayleigh difference for the uniform spanning measure on $G$ using the deleted-graph identity. We have 
\[
\tau(G-e)\tau(G-f)-\tau(G)\tau(G-e-f)=(\tau(G)u_e(f))^2,\]
which is a square and hence is nonnegative.

Now, we proceed with the quantity we will be concerned with in this section. For a connected graph $G$, we define the two-forest ratio as \[Q(G):=\dfrac{\tau_2(G)}{\tau(G)},\]
which is the ratio of two-component forests to spanning trees in a graph. The weighted form of the two-forest ratio has been defined by Pitman and Tang \cite{PitmanTang2018}, and the two-forest ratio $Q(G)$ we use was studied in \cite{richman2023counting,KasselKenyonWu2015}.   

If we expand the Rayleigh difference for the uniform forest measure using the deleted-graph framework we have introduced at large $\lambda$ fugacity, i.e. we are only interested in the highest two terms, we get the following two quantities, one for the trees which is the highest coefficient and the second for the two-component forests which is the second-highest coefficient. 

\begin{align*}
H_0&=\tau(G-e)\tau(G-f)-\tau(G)\tau(G-e-f)\\
H_1&=\tau_2(G-e)\tau(G-f)+\tau(G-e)\tau_2(G-f)-\tau_2(G)\tau(G-e-f)-\tau(G)\tau_2(G-e-f),
\end{align*}
where we have

\begin{align*}
\Rayleigh_{ef}(\ForestZ_G)(\lambda\mathbf1)
&=\frac{\ForestZ_{G-e}(\lambda)\ForestZ_{G-f}(\lambda)-\ForestZ_G(\lambda)\ForestZ_{G-e-f}(\lambda)}{\lambda^2}\\
&=\sum_{i=0}^{2n-4}H_i\lambda^{2n-4-i}\\
&=H_0\lambda^{2n-4}+H_1\lambda^{2n-5}+O(\lambda^{2n-6}).
\end{align*}

We once again note that if $H_0\neq 0$, then we already have the result for sufficiently large $\lambda$ since $\lambda^{2n-4}$ dominates the rest of the terms when $\lambda > C$ for some large constant $C$. Hence, we focus on the case when $H_0=0$. Since $\tau(G)\neq 0$, this means that $u_e(f)=0$ and in particular that $G-e-f$ is connected. Moreover, we have 
\[\tau(G-e)\tau(G-f)=\tau(G)\tau(G-e-f).\]

If we divide $H_1$ by $\tau(G)\tau(G-e-f)$ using the equality above, we get

\begin{align*}
\dfrac{H_1}{\tau(G)\tau(G-e-f)}=&\dfrac{\tau_2(G-e)\tau(G-f)}{\tau(G-e)\tau(G-f)}+\dfrac{\tau(G-e)\tau_2(G-f)}{\tau(G-e)\tau(G-f)}\\
&-\dfrac{\tau_2(G)\tau(G-e-f)}{\tau(G)\tau(G-e-f)}-\dfrac{\tau(G)\tau_2(G-e-f)}{\tau(G)\tau(G-e-f)}\\
=&Q(G-e)+Q(G-f)-Q(G)-Q(G-e-f).
\end{align*}

We will work on the deleted difference of the function $Q$ from here on. We use the following theorem giving an electrical perspective on the two-forest ratio $Q$ proved by Richman, Shokrieh, and Wu \cite[Theorem C]{richman2023counting}.

\begin{thm}(\cite[Theorem C]{richman2023counting})
\label{thm: richmanshokriehwu}
Let $G=(V,E)$ be a connected graph and fix a vertex $o\in V$. Then,
\[\dfrac{\tau_2(G)}{\tau(G)}=\dfrac{1}{4}\sum_{g\in E(G)}r_G(g)^2+\dfrac{3}{4}\sum_{g=xy\in E(G)}(r_G(x,o)-r_G(y,o))^2.\]
\end{thm}

We will also use the Sherman-Morrison formula which extends the matrix deletion lemma \cite{ShermanMorrison1950}. 

\begin{lemma}(Sherman-Morrison Formula)
\label{lem: shermanmorrison}
Suppose $A\in \mathbb{R}^{n\times n}$ is an invertible matrix and $u,v \in \mathbb{R}^n$ are column vectors. Then $A+uv^T$ is invertible if and only if $1+v^TA^{-1}u\neq 0$. In this case, 
\[(A+uv^{T})^{-1}=A^{-1}-\dfrac{A^{-1}uv^TA^{-1}}{1+v^TA^{-1}u}.\]
\end{lemma}

First, we tie the resistance in deleted graphs to the resistance and potential difference in the original graph using the fact that $u_e(f)=0$. The main importance of Lemma \ref{lem:electrical-calculus} is that since the potential difference across $e$ is zero when we fix $f$ and vice versa, we observe that the deleted graphs become independent, and the formulas turn into summation. 

\begin{lemma}
\label{lem:electrical-calculus}
Let $G=(V,E)$ be a connected graph, and let $e,f \in E$ be non-bridge edges. Assume $u_e(f)=0$. Then, for $h\in\{e,f\}$ we have
\[
L_{G-h}^\dagger=L^\dagger+
\frac{L^\dagger b_hb_h^TL^\dagger}{1-r_G(h)}.
\]
Consequently, for $g\notin\{e,f\}$, we have
\begin{align*}
r_{G-e}(g)&=r_G(g)+\frac{u_e(g)^2}{1-r_G(e)},\\
r_{G-e-f}(g)&=r_G(g)+\frac{u_e(g)^2}{1-r_G(e)}
+\frac{u_f(g)^2}{1-r_G(f)}.
\end{align*}
If $\widetilde v_e=v_e-v_e(o)\mathbf1$, so that $\widetilde v_e(o)=0$, and $s_e(g)=\widetilde v_e(g^+)+\widetilde v_e(g^-)$, then
\[
r_{G-e}(g^+,o)-r_{G-e}(g^-,o)
=r_G(g^+,o)-r_G(g^-,o)
+\frac{u_e(g)s_e(g)}{1-r_G(e)},
\]
with the analogous single- and double-edge deletion formulas for $f$.
\end{lemma}

\begin{proof}
Since $h$ is a non-bridge edge, we have $1-r_G(h)>0$. Now, on $\mathbb{1}^\perp$, the Laplacian is invertible, and we have $L_{G-h}=L-b_hb_h^T$. Using Lemma \ref{lem: shermanmorrison}, we have
\[(L-b_hb_h^T)^{-1}=L^{-1}+\dfrac{L^{-1}b_hb_h^TL^{-1}}{1-b_h^TL^{-1}b_h^T}.\]

On $\mathbf{1}^\perp$, when the Laplacian has an inverse, we already have $L^{-1}|_{\mathbf{1}^\perp}=L^{\dagger}|_{\mathbf{1}^\perp}$. On constants, we observe that both sides vanish above; hence they also agree on $\mathbf{1}^\perp$ which is how we can write the above with the Moore-Penrose pseudoinverse operator. 

If we multiply the above with $b_g$ on both sides, we get
\[r_{G-h}(g)=r_G(g)+
\dfrac{u_h(g)^2}{1-r_G(h)}.\]

Next, let us look at the identity where both $e$ and $f$ are deleted. We note that \[L^\dagger b_f=v_f+
\dfrac{u_e(f)}{1-r_G(e)}v_e=v_f,\] and similarly
\[r_{G-e}(f)=r_G(f)+\dfrac{u_e(f)^2}{1-r_G(e)}=r_G(f)\]
since we have $u_e(f)=0$.

Thus, deleting the edge $f$ from the graph $G-e$ would yield 
\begin{align*}
(L-b_eb_e^T-b_fb_f^T)^{-1}&=(L-b_eb_e^T)^{-1}+\dfrac{(L-b_eb_e^T)^{-1}b_fb_f^T(L-b_eb_e^T)^{-1}}{1-b_f^T(L-b_eb_e^T)^{-1}b_f}\\
&=L^\dagger+\dfrac{v_ev_e^T}{1-r_G(e)}+\dfrac{v_fv_f^T}{1-r_G(f)}
\end{align*}
which adds the term $\dfrac{u_f(g)^2}{1-r_G(f)}$ once we use the identities we established above for the Moore-Penrose pseudoinverse after multiplying with $b_g$ on both sides. Thus, this establishes
\[r_{G-e-f}(g)=r_G(g)+\frac{u_e(g)^2}{1-r_G(e)}
+\frac{u_f(g)^2}{1-r_G(f)}.\]

For the endpoints $g^+,g^-$ of an edge $g$, we do a similar analysis, except we set the grounding vertex with $\widetilde v_e=v_e-v_e(o)\mathbf1$ and $\widetilde v_f=v_f-v_f(o)\mathbf1$. We calculate each term $r_{G-e}(g^+,o)$ and $r_{G-e}(g^-,o)$ separately. Then we subtract them. The new term that is added is
\[
\frac{\widetilde v_e(g^+)^2-\widetilde v_e(g^-)^2}{1-r_G(e)}
=\frac{u_e(g)s_e(g)}{1-r_G(e)}.
\]
Deleting the edge $f$ is symmetric, and leads to the same except we switch the edges $e$ and $f$ in the appropriate places.  Because deletion of $e$ leaves $v_f$ and $r_G(f)$ unchanged, the double-deletion formula is the sum of these two updates. 
\end{proof}

Next, we use Lemma \ref{lem:electrical-calculus} to give a new way to express the two-forest ratio in deleted graphs. This will be the intermediate step leading us to the sums of squares form we want to prove.

\begin{lemma}
\label{lem:intermediate-electrical}
In the setting of Lemma~\ref{lem:electrical-calculus}, we have
\begin{align*}
&Q(G-e)+Q(G-f)-Q(G)-Q(G-e-f)\notag\\
&\quad=-\frac{1}{2(1-r_G(e))(1-r_G(f))}
\left(
\sum_{g\in E} u_e(g)^2u_f(g)^2
+3\sum_{g\in E} u_e(g)u_f(g)s_e(g)s_f(g)
\right).
\label{eq:intermediate-electrical}
\end{align*}
\end{lemma}

\begin{proof}
We will use Lemma \ref{lem: shermanmorrison} to write the quantities. First, let us consider the square parts that are not grounded at the vertex $o$. We have four terms $r_H(g)^2$ with $H$ ranging over $G,G-e,G-f,G-e-f$. For each of the deleted graphs, we use Lemma \ref{lem:electrical-calculus} to rewrite them. 
\begin{align*}
\dfrac{1}{4}(r_{G-e}(g)^2+r_{G-f}(g)^2-r_G(g)^2-r_{G-e-f}(g)^2)=&\dfrac{1}{4}\left(\left(r_G(g)+\dfrac{u_e(g)^2}{1-r_G(e)}\right)^2+\left(r_G(g)+\dfrac{u_f(g)^2}{1-r_G(f)}\right)^2\right)\\
&-\dfrac{1}{4}\left(r_G(g)^2+\left(r_G(g)+\dfrac{u_e(g)^2}{1-r_G(e)}+\dfrac{u_f(g)^2}{1-r_G(f)}\right)^2\right)\\
=&-\dfrac{1}{2}\dfrac{u_e(g)^2u_f(g)^2}{(1-r_G(e))(1-r_G(f))}.
\end{align*}

The edge $e$ is absent from the two graphs in which $e$ is deleted.  In the other two graphs its resistance is unchanged by deletion of $f$, because $u_f(e)=u_e(f)=0$; hence its net contribution is zero.  The same holds for $f$.  Thus the first part contributes
\begin{equation*}\label{eq:resistance-four-difference}
-\frac{1}{2(1-r_G(e))(1-r_G(f))}\sum_{g\in E(G)}u_e(g)^2u_f(g)^2.
\end{equation*}

Now, we look at the part which is grounded at the vertex $o$, that is, we choose a vertex $o$, and use $\widetilde v_h=v_h-v_h(o)\mathbf1$ for $h=e,f$. For $g\not \in \{e,f\}$, we again use Lemma \ref{lem:electrical-calculus} to rewrite the terms. Just for this proof, for simplicity, we introduce the following notation: 
\[
d_H(g)=r_H(g^+,o)-r_H(g^-,o)\quad(H\in\{G,G-e,G-f,G-e-f\}),
\qquad
p_g=\frac{u_e(g)s_e(g)}{1-r_G(e)},\quad
q_g=\frac{u_f(g)s_f(g)}{1-r_G(f)}.
\]
Then we have
\begin{align*}
\dfrac{3}{4}\left(d_{G-e}(g)^2+d_{G-f}(g)^2-d_G(g)^2-d_{G-e-f}(g)^2\right)=&\dfrac{3}{4}\left((d_G(g)+p_g)^2+(d_G(g)+q_g)^2\right)\\
&-\dfrac{3}{4}\left(d_G(g)^2+(d_G(g)+p_g+q_g)^2\right)\\
=&-\dfrac{3}{2}p_gq_g\\
=&-\dfrac{3}{2}\frac{u_e(g)s_e(g)}{1-r_G(e)}\frac{u_f(g)s_f(g)}{1-r_G(f)}.
\end{align*}

The contributions of $e$ and $f$ again vanish since deleting the edges does not change the relevant resistance differences, while deleting the edges removes their summands. Then, the second part contributes
\begin{equation*}\label{eq:potential-four-difference}
-\frac{3}{2(1-r_G(e))(1-r_G(f))}
\sum_{g\in E(G)}u_e(g)u_f(g)s_e(g)s_f(g).
\end{equation*}
Adding the two parts above proves the desired equality.
\end{proof}

Next, we will prove that the quantities at hand are invariant under changes to the additive constants we choose. 

\begin{lemma}
\label{lem:gauge-integration}
Let $G=(V,E)$ be a connected graph and let $e,f\in E$ be non-bridge edges. Assume $u_e(f)=0$, and let 
\[f(G):=\sum_{g\in E}u_e(g)^2u_f(g)^2+3\sum_{g\in E}u_e(g)u_f(g)s_e(g)s_f(g).\]
Then, $f(G)$ is independent of the constants we choose for $\widetilde v_e$ and $\widetilde v_f$. 
\end{lemma}

\begin{proof}
The potential differences above are already invariant under changing additive constants since we subtract the potentials. We will be concerned with the part that contains $s_e(g)$ and $s_f(g)$. These by themselves are not invariant under changing additive constants; they are sums of potentials of the endpoints. Suppose we replace $\widetilde v_e$ by $\widetilde v_e+c_1$ and $\widetilde v_f$ by $\widetilde v_f+c_2$ for two distinct constants $c_1$ and $c_2$. Then the difference between $f(G)$ under the choice of $(\widetilde v_e,\widetilde v_f)$ and $f(G)$ under the choice of $(\widetilde v_e+c_1,\widetilde v_f+c_2)$ is
\[6c_1\sum_{g\in E}u_e(g)u_f(g)s_f(g)+6c_2\sum_{g\in E}u_e(g)u_f(g)s_e(g)+12c_1c_2\sum_{g\in E}u_e(g)u_f(g).\]

We prove that each sum above is actually $0$ under the assumptions we have. Recall that the Laplacian can be written as $L=\sum_{g\in E}b_gb_g^T$, and that $u_h(g)=b_g^T\widetilde v_h$. Thus, we have 
\[\sum_{g\in E}u_e(g)u_f(g)=
\sum_{g\in E}b_g^T\widetilde v_eb_g^T\widetilde v_f=\sum_{g\in E}\widetilde v_f^Tb_gb_g^T\widetilde v_e=\widetilde v_f^TL\widetilde v_e=b_f^T\widetilde v_e=u_e(f)=0.\]

For the other two, we do a similar analysis using the Laplacian. We write $\widetilde v_e^2$ for the entrywise square of the potentials. With this, we can write $u_e(g)s_e(g)=b_g^T(\widetilde v_e^2)$. Thus, we have
\[\sum_{g\in E}u_e(g)u_f(g)s_e(g)=\sum_{g\in E}b_g^T(\widetilde v_e^2)b_g^T\widetilde v_f=(\widetilde v_e^2)^TL\widetilde v_f=u_e(f)s_e(f)=0.\]
 The third term is obtained by interchanging $e$ and $f$ in the calculation above. Thus, all the terms are $0$, and the function $f(G)$ is independent of the additive constants in the potentials. 
\end{proof}

Lemma \ref{lem:gauge-integration} gives us the freedom to choose any additive constants we want. We use this freedom to choose the constants as in Theorem \ref{thm:leading-coefficients} to get the sums of squares form we want. 

\begin{lemma}
\label{lem: special-gauge-constants}
Let $G=(V,E)$ be a connected graph and let $e,f$ be non-bridge edges, and suppose $u_e(f)=0$. Let $f(G)$ be defined as in Lemma \ref{lem:gauge-integration}. If we choose $\widetilde{v}_e$ and $\widetilde{v}_f$ such that 
\[\widetilde v_e(f^-)=\widetilde v_e(f^+)=0,\qquad
\widetilde v_f(e^-)=\widetilde v_f(e^+)=0,\]
we have
\[-\dfrac{1}{2}f(G)=\sum_{xy\in E}(\widetilde{v}_e(x)\widetilde{v}_f(y)-\widetilde{v}_e(y)\widetilde{v}_f(x))^2.\]
\end{lemma}

\begin{proof}
For every $g\in V$, choose the additive constants by setting
\[\widetilde{v}_e(g):=v_e(g)-v_e(f^+)
\qquad\text{and}\qquad
\widetilde{v}_f(g):=v_f(g)-v_f(e^+).\]
Then $\widetilde{v}_e(f^+)=0$. Since
\[u_e(f)=\widetilde{v}_e(f^+)-\widetilde{v}_e(f^-)=0,\]
we also have $\widetilde{v}_e(f^-)=0$. Similarly, using 
\[u_f(e)=u_e(f)=0,\]
we have
\[\widetilde{v}_f(e^-)=\widetilde{v}_f(e^+)=0.\]
Thus the required choices of additive constants are well-defined. In what
follows, $s_e$ and $s_f$ are computed using these representatives of the
potentials. For brevity within this proof, we write
\[\phi_g:=\widetilde{v}_e(g),
\qquad \psi_g:=\widetilde{v}_f(g)\qquad (g\in V).\]
For every edge $xy\in E$, we have
\begin{align*}
u_e(xy)^2u_f(xy)^2
&=(\phi_x-\phi_y)^2(\psi_x-\psi_y)^2,\\
u_e(xy)u_f(xy)s_e(xy)s_f(xy)
&=(\phi_x^2-\phi_y^2)(\psi_x^2-\psi_y^2).
\end{align*}
We will use the following algebraic identity which follows by expanding both sides:
\begin{align*}
&(\phi_x-\phi_y)^2(\psi_x-\psi_y)^2
+3(\phi_x^2-\phi_y^2)(\psi_x^2-\psi_y^2)
+2(\phi_x\psi_y-\phi_y\psi_x)^2\\
&\qquad =
2(\phi_x\psi_x^2-\phi_y\psi_y^2)(\phi_x-\phi_y)
+2(\phi_x^2\psi_x-\phi_y^2\psi_y)(\psi_x-\psi_y).
\end{align*}

If we sum over all $xy\in E$ and use the definition of $f(G)$, we have
\begin{align*}
f(G)+2\sum_{xy\in E}(\phi_x\psi_y-\phi_y\psi_x)^2
&=2\sum_{xy\in E}(\phi_x\psi_x^2-\phi_y\psi_y^2)(\phi_x-\phi_y)\\
&\quad+2\sum_{xy\in E}(\phi_x^2\psi_x-\phi_y^2\psi_y)(\psi_x-\psi_y).
\end{align*}

Recall that for any two functions $\eta,\zeta\colon V\to\mathbb R$, the Laplacian satisfies
\[\sum_{xy\in E}(\eta(x)-\eta(y))(\zeta(x)-\zeta(y))=\eta^TL\zeta.\]
Therefore,
\[f(G)+2\sum_{xy\in E}(\phi_x\psi_y-\phi_y\psi_x)^2=2(\phi\psi^2)^TL\phi+2(\phi^2\psi)^TL\psi,\]
where the products $\phi\psi^2$ and $\phi^2\psi$ are taken coordinate-wise.

Adding a constant to a potential does not change its Laplacian, since $L\mathbf 1=0$. Thus, we can write
\[L\phi=b_e\qquad\text{and}\qquad L\psi=b_f.\]
Using these, we rewrite the sum above as
\begin{align*}
f(G)+2\sum_{xy\in E}(\phi_x\psi_y-\phi_y\psi_x)^2&=2(\phi\psi^2)^Tb_e+2(\phi^2\psi)^Tb_f\\
&=2(\phi_{e^+}\psi_{e^+}^2-\phi_{e^-}\psi_{e^-}^2)\\
&\quad+2(\phi_{f^+}^2\psi_{f^+}-\phi_{f^-}^2\psi_{f^-}).
\end{align*}
By our choice of additive constants,
\[\psi_{e^+}=\psi_{e^-}=0\qquad\text{and}\qquad \phi_{f^+}=\phi_{f^-}=0.\]
Hence both terms on the right-hand side vanish, and therefore
\[f(G)=-2\sum_{xy\in E}(\phi_x\psi_y-\phi_y\psi_x)^2.\]
Substituting back $\phi_g=\widetilde{v}_e(g)$ and
$\psi_g=\widetilde{v}_f(g)$ gives
\[-\dfrac{1}{2}f(G)=\sum_{xy\in E}
\left(\widetilde{v}_e(x)\widetilde{v}_f(y)-\widetilde{v}_e(y)\widetilde{v}_f(x)\right)^2,\]
as required.
\end{proof}

At last, we come to the crux of our argument. We use Lemma \ref{lem: special-gauge-constants} to prove that the two-forest ratio differences can be expressed as sums of squares which will be the key to proving the large $\lambda$ fugacity result. 

\begin{lemma}
\label{lem:sums-of-squares}
Let $G=(V,E)$ be a connected graph, and let $e,f$ be non-bridges. Assume that $u_e(f)=0$, and let $Q(G)$ be the two-forest ratio $\dfrac{\tau_2(G)}{\tau(G)}$. Let $\widetilde{v}_e,\widetilde{v}_f$ be potentials as in Lemma \ref{lem: special-gauge-constants}, and set $w_{ef}(x,y):=\widetilde{v}_e(x)\widetilde{v}_f(y)-\widetilde{v}_e(y)\widetilde{v}_f(x)$. We have
\[Q(G-e)+Q(G-f)-Q(G)-Q(G-e-f)=\dfrac{1}{(1-r_G(e))(1-r_G(f))}\sum_{xy\in E}w_{ef}(x,y)^2.\]
\end{lemma}

\begin{proof}
By Lemma \ref{lem:intermediate-electrical}, we have
\[Q(G-e)+Q(G-f)-Q(G)-Q(G-e-f)=-\dfrac{f(G)}{2(1-r_G(e))(1-r_G(f))}.\]
Moreover, by Lemma \ref{lem: special-gauge-constants}, we have 
\[f(G)=-2\sum_{xy\in E}w_{ef}(x,y)^2.\]
Combining the two, we get the desired equality. 
\end{proof}

We are now ready to prove the main aim of this section, negative edge correlation for the arboreal gas measure at large $\lambda$ fugacity.

\begin{proof}(Proof of Corollary \ref{cor:large-fugacity}.)
If $H_0>0$, then we choose $C_G\geq1$ such that whenever $\lambda>C_G$, we have $H_0 \lambda>\sum_{i=1}^{2n-4}|H_i|$. Then we get the positivity of the Rayleigh difference. Now, suppose $H_0=0$, and we have $H_1>0$. In a similar manner, we choose $C_G\geq1$ sufficiently large so that the largest term $H_1\lambda^{2n-5}$ dominates the rest, i.e. $H_1 \lambda>\sum_{i=2}^{2n-4}|H_i|$. Thus, we again have positivity of the Rayleigh difference. Otherwise, we have $H_1=0$. This means that $\widetilde{v}_e(x)\widetilde{v}_f(y)=\widetilde{v}_e(y)\widetilde{v}_f(x)$.
\end{proof} 

\subsection{Equality Analysis and Examples}
\label{sec:vanishing}
In this section, we further analyse the equality case in the sums of squares expression we proved in Lemma \ref{lem: special-gauge-constants}. It turns out that the equality admits a nice decomposition in terms of graphs and electrical parameters. We observe that after fixing the edges $e,f$ which are non-bridges, the only way $H_1$ vanishes is that 
\[\widetilde{v}_e(x)\widetilde{v}_f(y)=\widetilde{v}_e(y)\widetilde{v}_f(x)\]
holds for all edges $g=xy\in E$. 

We will prove that if this is the case, then the graph $G$ needs to be a union of two distinct graphs one containing the edge $e$, call it $G_e$, the other containing the edge $f$, call it $G_f$, glued on a vertex set $S$ such that current on $S$ viewed from the perspective of the graph $G_e$ and from the perspective of the graph $G_f$ are both $0$. This introduces a very special graph class which fails to satisfy the large $\lambda$ fugacity result. We do not mean here that negative edge correlation fails; we only mean that on this class, to prove negative edge correlation one has to investigate the properties of $H_2$ or in general the Rayleigh difference. We illustrate $G_e,G_f$ and $S$ in Figure \ref{fig:equipotential-decomposition}, and then state the theorem.

\begin{figure}[htbp]
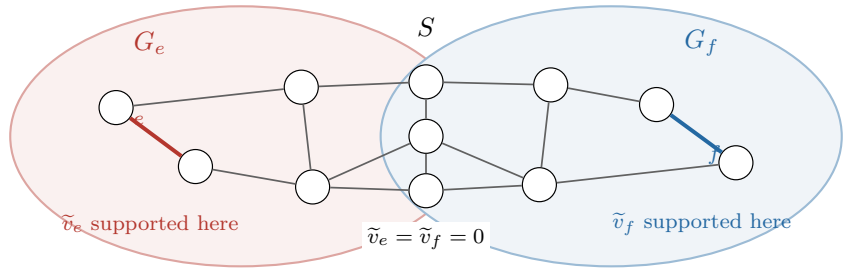

\centering
\EquipotentialFigure
\caption{The equipotential decomposition.  The $e$-current potential is supported on the red side, the $f$-current potential on the blue side, and both vanish on the common interface $S$.  Edges induced by $S$ belong to the amalgamation only once.}
\label{fig:equipotential-decomposition}
\end{figure}

\begin{thm}
\label{thm:equipotential}
Let $G=(V,E)$ be a connected graph, and let $e,f \in E$ be non-bridge edges. We have $H_0=H_1=0$ if and only if $G$ can be written as the sum of two graphs 
\[G=G_e \cup_S G_f, \quad e\in E(G_e), \quad f\in E(G_f),\]
such that for any $g,h\in S$, we have $v_e(g)=v_e(h)$ and $v_f(g)=v_f(h)$ where the potentials are taken in their respective graphs $G_e$ and $G_f$. 
Here $\cup_S$ denotes that we glue the two graphs on the vertex set $S$ with $E(G[S])$ counted only once.
\end{thm}

\begin{proof}
The sums of squares form from Lemma \ref{lem:sums-of-squares} shows that $H_0=H_1=0$ if and only if under the choice of special additive constants we have
\[\widetilde{v}_e(x)\widetilde{v}_f(y)=\widetilde{v}_e(y)\widetilde{v}_f(x), \quad \text{for every} \;\; xy\in E(G).\]

We partition the vertices into the following four sets: 
\begin{align*}
V_e&=\{x:\widetilde{v}_e(x)\neq0,\ \widetilde{v}_f(x)=0\},&
V_f&=\{x:\widetilde{v}_e(x)=0,\ \widetilde{v}_f(x)\neq0\},\\
S&=\{x:\widetilde{v}_e(x)=\widetilde{v}_f(x)=0\},&
V_{ef}&=\{x:\widetilde{v}_e(x)\widetilde{v}_f(x)\neq0\}.
\end{align*}

We observe that $\widetilde{v}_e(x)\widetilde{v}_f(y)=\widetilde{v}_e(y)\widetilde{v}_f(x)$ forbids edges from $V_e$ to $V_f\cup V_{ef}$ and similarly from $V_f$ to $V_e\cup V_{ef}$. We define the vertex boundary of a set $A\subset V(G)$ as $\partial_V A:=\{y \in V\setminus A: \exists x\in A \;\; \text{with} \;\; xy\in E\}$. Then, we note that $\partial_V(V_{ef})\subseteq S$ since there are no edges from $V_{ef}$ to either $V_e$ or $V_f$. We also note that by our choice of potentials, none of the endpoints of $e$ or $f$ belong to $V_{ef}$. Thus, the potentials there are harmonic functions. Then, by the maximum principle along the boundary, we observe that the potentials $\widetilde{v}_{e}$ and $\widetilde{v}_f$ must vanish on $V_{ef}$ which contradicts the definition of $V_{ef}$. Thus, we have $V_{ef}=\emptyset$, and the partition of the vertex set as $V(G)=V_e\sqcup S \sqcup V_f$. 

We set $G_e:=G[V_e\sqcup S]$, and $G_f:=G[V_f\sqcup S]$. Observe that every edge lies in either of the graphs $G_e$ or $G_f$ by construction or is in the induced graph $G[S]$, and by definition we have that the potentials induced by the sets $G_e$ and $G_f$ on $S$ are all $0$, or equi-potentials if we change the additive constant. 

Conversely, suppose there is a partitioning with the vertices in $S$ having equi-potentials. We normalize each side's potential to $0$ on $S$ by just shifting the additive constant and extend it to vertices in the other side by just setting $\widetilde v_f(x)=0$ for all $x\in V(G_e)$ and $\widetilde v_e(x)=0$ for all $x\in V(G_f)$. All the conditions are satisfied, and each summand in the formula vanishes by construction.
\end{proof}

It is important to note that Theorem \ref{thm:equipotential} does not state that the entire Rayleigh difference vanishes. Even though the second coefficient may vanish, the rest need not. A two-vertex gluing due to Wagner makes this apparent. The following is the graphical form of Wagner's independence set Rayleigh difference across a matroid $2$-sum \cite[Corollary 5.7]{wagner2008negatively}. Thus, for completeness, we present the proof of this special case.

\begin{prop}
\label{prop: wagner-2-sums}
Let $G=G_1\cup_{x,y} G_2$ have disjoint edge sets, with $e\in E(G_1)$ and $f\in E(G_2)$. Let $a_i$ and $b_i$ be the forest generating polynomials of $G_i$ restricted to $x\leftrightarrow y$ and $x\not \leftrightarrow y$ respectively. Moreover, we let $a_1^e,b_1^e$ and $a_2^f,b_2^f$ impose in addition that the indicated edge is present in the configurations. Then, we have 
\[\lambda^2\Rayleigh_{ef}(\ForestZ_G)(\lambda\mathbf1)=(a_1^eb_1-b_1^ea_1)(a_2^fb_2-b_2^fa_2).\]
\end{prop}

\begin{proof}
A forest of $G$ restricts to forests on both of the graphs $G_1$ and $G_2$. Only one can connect $x$ to $y$; otherwise in the union we would have a cycle passing through $x$ and $y$ since the edge sets of $G_1$ and $G_2$ are disjoint. Thus, we can write the partition function for the arboreal gas measure as 
\[\ForestZ_G=a_1b_2+a_2b_1+b_1b_2.\]
For $S\subseteq E(G)$, let $Z_G^{\mathcal F}[S]:=\sum_{F\in\mathcal F(G):\,S\subseteq F}\lambda^{|F|}$, retaining all edge weights. We also have 
\begin{align*}
Z_G^{\mathcal F}[\{e\}]&=a_1^eb_2+b_1^ea_2+b_1^eb_2,\\
Z_G^{\mathcal F}[\{f\}]&=a_1b_2^f+b_1a_2^f+b_1b_2^f,\\
Z_G^{\mathcal F}[\{e,f\}]&=a_1^eb_2^f+b_1^ea_2^f+b_1^eb_2^f,
\end{align*}
Here $Z_G^{\mathcal F}[\{e\}],Z_G^{\mathcal F}[\{f\}],Z_G^{\mathcal F}[\{e,f\}]$ indicate the forest sums with the appropriate edges present. 

Now, if we just calculate $Z_G^{\mathcal F}[\{e\}]Z_G^{\mathcal F}[\{f\}]-Z_G^{\mathcal F}[\{e,f\}]\ForestZ_G$, we get the desired equality after noting that we have extra $\lambda$ terms coming from $e$ and $f$ not present in the Rayleigh difference. 
\end{proof}

We end this section by giving several examples which establish the strictness and the equality in $H_1$. The examples are each of a different flavor, and we include each one to show how the two currents might interact with each other when, by symmetry, we get $H_0=0$. In each example, we give the currents through $e$, and establish $u_e(f)$. We start with the most common example which is the complete graph, and continue with complete joins, torus grid graphs, Petersen graphs and $K_4$ blocks which will be of special interest in the subsequent part. 

\begin{example}
Let $K_n$ denote the complete graph on $n\geq 4$ vertices $V=[n]:=\{1,2,\ldots,n\}$. Let $e=12$ and $f=34$ be disjoint edges. We inject unit current through $e^+$ ($f^+$ when we work with $f$) and extract it through $e^-$ ($f^-$ when we work with $f$). By symmetry, we have $u_e(f)=0$, and the special potential $\widetilde{v}_e$ we established takes values $\frac{1}{n}$ and $-\frac{1}{n}$ on vertices $1$ and $2$ respectively while vanishing at every other vertex. The potential $\widetilde{v}_f$ is similar except it takes values on vertices $3,4$ and vanishes otherwise. One can calculate these values using Kirchhoff's law and noting that by symmetry every vertex except the source and the sink must have the same potential. Thus, in the sums of squares formula we have, only four terms remain nonzero, which are the edges $13,14,23,24$ comprised of the endpoints of $e$ and $f$. Thus, we get
\[H_1=\tau(K_n)^2\dfrac{4}{n^4}=4n^{2n-8}>0.\]

If $e,f$ are adjacent edges, we observe that $H_0>0$; thus we have the large $\lambda$ fugacity result immediately. 
\end{example}

We consider next the complete joins. Let $G_1,\ldots, G_r$ be disjoint graphs. We define the complete join $G:=G_1\lor \ldots \lor G_r$ to be the graph with vertex set $\bigcup_{i\in [r]} V(G_i)$ and the edge set $\bigcup_{i\in [r]} E(G_i)\cup \bigcup_{i\neq j \in [r]}\{xy: x\in G_i,y\in G_j\}$, i.e. we put all edges across the parts and keep the edges inside the graphs $G_i$. For the next three examples, we add the following Figure \ref{fig:symmetric-examples} to show the edges $e,f$ and the symmetry in the graphs presented. 

\begin{figure}[htbp]
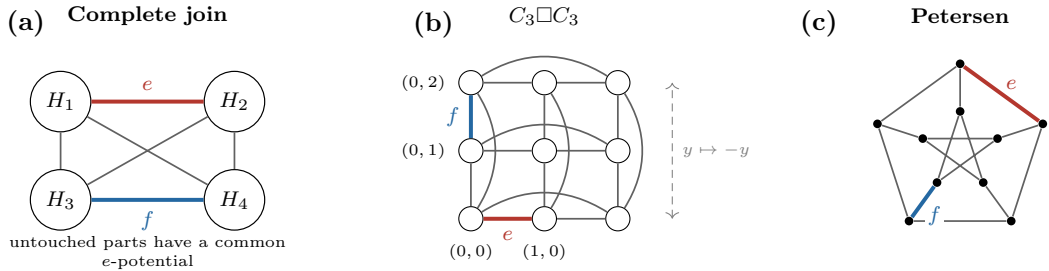

\centering
\SymmetricExamplesFigure
\caption{Three ways symmetry or a complete join can annihilate the leading transfer-current term.  In every part $e$ is red and $f$ is blue.  Part (a) shows the four relevant parts of a complete join.  Part (b) shows the edges $e,f$ and the reflection on $C_3\square C_3$.  Part (c) shows the edges $e,f$ in the Petersen graph.}
\label{fig:symmetric-examples}
\end{figure}

\begin{example}
Let $G=G_1\lor \ldots \lor G_r$ with $r\geq 4$. Let $e=xy$ with $x\in G_1,y\in G_2$, and let $f=st$ with $s\in G_3,t\in G_4$ be two disjoint edges. Let $v$ be the potential induced by injecting one unit of current through $e^+$ and extracting it from $e^-$, and let $x\in G_i$ with $i\geq 3$. Since $x$ is not a source or a sink vertex, we will use the harmonicity of the potential function with the fact that $G_i$ is adjacent to all the vertices outside $G_i$ to write 
\[(L_{G_i}+(|V(G)|-|V(G_i)|)I)v|_{G_i}=\left(\sum_{y\not \in G_i}v(y)\right)\mathbf{1}.\]

The matrix $(L_{G_i}+(|V(G)|-|V(G_i)|)I)v|_{G_i}$ is positive definite and sends constant vectors to constant vectors; thus $v$ is constant on $G_i$. With this, we have \[v(x)=\dfrac{1}{|V(G)|}\sum_{y \in V(G)}v(y).\]

Thus, in particular, we observe that $u_e(f)=0$ since $f$ has endpoints in parts $G_3$ and $G_4$. If we subtract this potential value, we observe that $\widetilde v_e$ vanishes on $G_3\cup G_4$. The same analysis can be made to produce a potential $\widetilde v_f$ which vanishes on $G_1\cup G_2$, especially on the endpoints of $e$. 

We know that the potential $\widetilde v_e$ has distinct values at the endpoints of $e=xy$ and $\widetilde v_f$ has distinct values at the endpoints of $f=st$. We can choose the potentials so that $\widetilde v_e$ has a nonzero value on $G_1\cup G_2$ and $\widetilde v_f$ has a nonzero value on $G_3\cup G_4$. Now, the complete join contains all the edges across the endpoints $x,y,s,t$. Thus, there exists a nonzero value in the sums of squares formula we have, which in turn establishes \[H_1/\tau(G)^2=\sum_{cd\in E(G)} w_{ef}(c,d)^2\geq w_{ef}(x,s)^2+w_{ef}(x,t)^2+w_{ef}(y,s)^2+w_{ef}(y,t)^2>0.\]

We note that this argument includes multipartite graphs with edges $e,f$ having endpoints in four distinct parts. 
\end{example}

Before proceeding with the $K_4$ blocks, we will do two concrete examples, one on the torus grid graph $C_3 \square C_3$, where $\square$ denotes the Cartesian product of two graphs, i.e. $\{(x_1,y_1),(x_2,y_2)\}\in E(G_1\square G_2)$ if $x_1x_2\in E(G_1)$ or $y_1y_2\in E(G_2)$, and the second one for the Petersen graph. 

\begin{example}
In $C_3 \square C_3$, we take $e=\{(0,0),(1,0)\}$ and $f=\{(0,1),(0,2)\}$. The reflection defined as $(x,y) \to (x,-y)$ fixes the orientation of the edge $e$ and reverses the orientation of the edge $f$. Thus, we must have $u_e(f)=0$. We write the potentials $\widetilde{v}_e$ and $\widetilde{v}_f$ as follows, with the rows indexed by the first coordinate and the columns indexed by the second coordinate:
\[18\widetilde v_e=
\begin{pmatrix}-3&0&0\\5&2&2\\1&1&1\end{pmatrix},
\qquad
18\widetilde v_f=
\begin{pmatrix}0&-4&4\\0&-1&1\\0&-1&1\end{pmatrix}.\]

We can see that the first potential vanishes at the endpoints of $f$ and the second vanishes at the endpoints of $e$. If we calculate the sums of squares formula we had, we get
\[\sum_{xy\in C_3\square C_3}(\widetilde{v}_e(x)\widetilde{v}_f(y)-\widetilde{v}_e(y)\widetilde{v}_f(x))^2=\frac{29}{5832},\]
and looking at any cofactor of the Laplacian yields \[\tau(C_3\square C_3)=\operatorname{det}(L|_{8\times 8})=11664.\]

Hence, we have $H_1=11664^2\frac{29}{5832}=676512$. We do not go into details, but the same reflection works in $C_m\square C_{2k+1}$ for the corresponding horizontal and antipodal vertical edges. 

For the Petersen graph $G(5,2)$, we label the outer pentagon with $\{0,\dots,4\}$ and the inner pentagon with $\{0',\ldots,4'\}$ and join the vertices $i'$ to $(i+2)'$ modulo $5$. Let $e=01$ be oriented from $0$ to $1$, and let $f=33'$ be oriented from $3$ to $3'$. Then in the order $(0,\ldots,4,0',\ldots,4')$, the potentials become 
\[10\widetilde{v}_e=(-3,3,1,0,-1,-1,1,0,0,0),
\qquad
10\widetilde{v}_f=(0,0,-1,-3,-1,1,1,0,3,0).\]

We observe that $u_e(f)=0$. If we sum over the sums of squares formula we have, we get 
\[\sum_{xy\in G(5,2)}(\widetilde{v}_e(x)\widetilde{v}_f(y)-\widetilde{v}_e(y)\widetilde{v}_f(x))^2=\frac{9}{1250},\]
and the Laplacian cofactor yields 
\[\operatorname{det}(L|_{9\times 9})=2000.\]
Thus, we have $H_1=2000^2\frac{9}{1250}=28800$. 
\end{example}

\begin{example}
\label{ex:k4-chain}
Let $B_n$ be the graph on the $2n$ vertices $x_1,\ldots,x_n,y_1,\ldots,y_n$ with the edge set \[\bigcup_{i=1}^{n-1}\binom{\{x_i,y_i,x_{i+1},y_{i+1}\}}{2}.\] Let $e=x_1y_1$ and $f=x_ny_n$. Hence $B_n$ is a chain of $K_4$'s each glued along two vertices and the distinguished edges $e$ and $f$ are the first and the last vertical edges as can be seen in Figure \ref{fig:k4-chain}.

\begin{figure}[htbp]
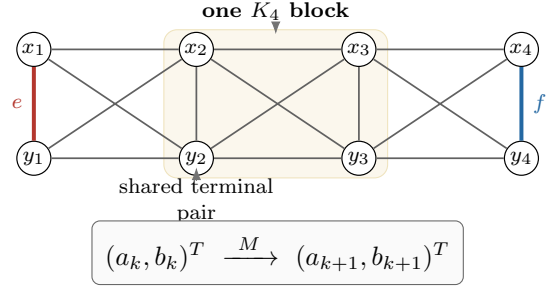

\centering
\KFourChainFigure
\caption{The $K_4$ chain $B_4$.  Each consecutive four-set induces a copy of $K_4$, and consecutive blocks share a terminal edge. Passing across one block sends the connected/disconnected terminal-state vector $(a_k,b_k)^T$ to $M(a_k,b_k)^T$.}
\label{fig:k4-chain}
\end{figure}

We use Proposition \ref{prop: wagner-2-sums} to show that
\[\lambda^2\Rayleigh_{ef}(\ForestZ_{B_n})(\lambda\mathbf1)=2^{n-1}\lambda^{2n}(1+2\lambda)^{n-1}.\]

Let $B_k$ denote the chain on $\{x_1,\ldots,x_k,y_1,\ldots,y_k\}$, and let $a_k,b_k$ denote the generating polynomials of forests on $B_k$ where $x_k \leftrightarrow y_k$ and $x_k \not \leftrightarrow y_k$ respectively. Similarly, let $a_k^e,b_k^e$ denote the generating polynomials for which the edge $e$ is present. Let \[w_k:=\binom{a_k}{b_k}, \quad w_k^e:=\binom{a_k^e}{b_k^e}.\] We will examine how these change when we add a new $K_4$. For $a_{k+1}$, we'll first look at what happens when $x_k\leftrightarrow y_k$. Suppose first that $x_{k+1}y_{k+1}$ is present in the forest. Then we can pick either one edge from the crossing $4$ edges or none, since we need to preserve acyclicity and $x_k\leftrightarrow y_k$. Thus, there is a multiplicative factor of $\lambda(1+4\lambda)$ to the existing configurations. Now, suppose $x_{k+1}y_{k+1}$ is absent. Then, to get to $x_{k+1}\leftrightarrow y_{k+1}$, we need exactly one edge going to each vertex; this leads to four choices with weights $4\lambda^2$. Hence, from connected $a_k$ to connected $a_{k+1}$ we have a multiplicative factor of $\lambda+8\lambda^2$. A similar analysis leads to the following matrix: 
\begin{equation*}\label{eq:k4-transfer}
M=\begin{pmatrix}
\lambda(1+8\lambda)&\lambda(1+2\lambda)(1+4\lambda)\\
1+4\lambda&(1+2\lambda)^2
\end{pmatrix}.
\end{equation*}

The other three entries of the matrix are calculated in the same way and arise from the following table:
\[\begin{array}{c|cc}
& x_k\leftrightarrow y_k
& x_k\not\leftrightarrow y_k\\ \hline
x_{k+1}\leftrightarrow y_{k+1}
& \lambda+8\lambda^2
& \lambda+6\lambda^2+8\lambda^3\\[1mm]
x_{k+1}\not\leftrightarrow y_{k+1}
& 1+4\lambda
& 1+4\lambda+4\lambda^2
\end{array}\]

This leads to the transformation $Mw_k=w_{k+1}$, and $Mw_k^e=w_{k+1}^e$ with $\operatorname{det}(M)=2\lambda^2(1+2\lambda)$.
Let $W_k=a_k^eb_k-b_k^ea_k=\operatorname{det}\left(\begin{pmatrix}
a_k^e & a_k \\
b_k^e & b_k
\end{pmatrix}\right)$. We have $\operatorname{det}(M)W_k=W_{k+1}$, and we also observe that $W_1=\lambda$ since $w_1=\binom{\lambda}{1}$ and $w_1^e=\binom{\lambda}{0}$. Thus, we have \[W_k=\operatorname{det}(M)^{k-1}\lambda.\] 
Applying the recurrence $n-1$ times gives
\[W_n=\lambda\operatorname{det}(M)^{n-1}=\lambda(2\lambda^2(1+2\lambda))^{n-1}.\]
We apply Proposition \ref{prop: wagner-2-sums}, by considering the
decomposition \[B_n=(B_n-f)\mathbin{\amalg_{\{x_n,y_n\}}}\{f\}.\]
We claim that $W_n$ is also the first factor associated with $B_n-f$.
Indeed, let $\bar{a}_n,\bar{b}_n,\bar{a}_n^e,\bar{b}_n^e$ denote the corresponding generating polynomials for $B_n-f$. Adding $f$ gives $a_n=\bar{a}_n+\lambda\bar{b}_n$, $b_n=\bar{b}_n$, $a_n^e=\bar{a}_n^e+\lambda\bar{b}_n^e$, and
$b_n^e=\bar{b}_n^e$.
Therefore,
\[W_n=a_n^eb_n-b_n^ea_n=(\bar{a}_n^e+\lambda\bar{b}_n^e)\bar{b}_n
-\bar{b}_n^e(\bar{a}_n+\lambda\bar{b}_n)=\bar{a}_n^e\bar{b}_n-\bar{b}_n^e\bar{a}_n.\]
Thus, adding $f$ does not change the first factor in
Proposition~\ref{prop: wagner-2-sums}. The graph consisting only of
$f$ contributes the second factor $\lambda$. Thus, we have
\[\begin{aligned}
\lambda^2\Rayleigh_{ef}(\ForestZ_{B_n})(\lambda\mathbf1)=\lambda W_n=\lambda(2\lambda^2(1+2\lambda))^{n-1}\lambda=2^{n-1}\lambda^{2n}(1+2\lambda)^{n-1}.
\end{aligned}\]
\end{example}

\section{Random-Cluster Measure with \texorpdfstring{$0<q<1$}{0<q<1}}
\label{sec:complete-proof}
In this section, we will prove Theorem \ref{thm:complete-graph} establishing negative edge correlation of the edges under the random-cluster model when the weights are all equal. The proof method will make it apparent why we don't allow arbitrary weights. 

We recall the bridge identity we established in Section \ref{sec: preliminarylemmas}, Proposition \ref{prop:bridge-identity}. Let $G=(V,E)$ be a connected graph, let $e=ab$ and $f=cd$ be distinct edges, and let $\RCZ_G$ be the random-cluster model with $q\neq 1$. Then, we have 
\[\dfrac{\Rayleigh_{ef}(\RCZ_G)}{(1-q)^2}=\Rayleigh_{ef}^b+\dfrac{Z_{ef}(J_{ef})Z_{ef}}{1-q},\]
where $J_{ef}$ is tested on $A\setminus\{e,f\}$ as defined in Section~\ref{sec: preliminarylemmas}, and
\[
Z_{ef}(\mathcal E)=\sum_{\substack{A\subseteq E(G)\\e,f\in A}}
\mathbf1_{\mathcal E}(A)q^{\kappa(A)}\prod_{g\in A\setminus\{e,f\}}z_g.
\]
In particular, $Z_{ef}=\partial_{ef}\RCZ_G$ removes the two marked-edge weights.

We will prove the main theorem by comparing weights among configuration classes on $K_n$. We fix a common weight $\lambda\geq 1$, so that for all $g\in E$, we have $z_g=\lambda$. Since the Rayleigh difference is independent of $z_e,z_f$, we set these two weights to $1$ throughout the proof. For $t\geq 1$, we first consider the spanning subgraphs with $t$ components, where $e,f$ are present. Over these configurations $A$, the contribution from $\prod_{g\in A}z_g$ becomes $\lambda^{|A\setminus \{e,f\}|}$. Define the following weighted sums for this configuration class which partition the configurations over $\{e,f\}$ pair being bridge-bridge, bridge-non-bridge, and joint bridges: 
\begin{align*}
a_t&:=\sum_{A: \kappa(A)=t}\lambda^{|A\setminus \{e,f\}|}\\
p_t^e&:=\sum_{\substack{A: \kappa(A)=t\\ e \: \text{is a bridge}}}\lambda^{|A\setminus \{e,f\}|}\\
p_t^f&:=\sum_{\substack{A: \kappa(A)=t\\ f \: \text{is a bridge}}}\lambda^{|A\setminus \{e,f\}|}\\
b_t&:=\sum_{\substack{A: \kappa(A)=t\\ e,f \: \text{are bridges}}}\lambda^{|A\setminus \{e,f\}|}\\
j_t&:=\sum_{\substack{A: \kappa(A)=t\\ e,f \: \text{are joint bridges}}}\lambda^{|A\setminus \{e,f\}|}.
\end{align*}

Using these, we can define the corresponding weights each class gets under the random-cluster model with connectivity parameter $q$:
\[A(q)=\sum_ta_tq^t,\quad P_e(q)=\sum_tp_t^eq^t,\quad P_f(q)=\sum_tp_t^fq^t,\quad B(q)=\sum_tb_tq^t,\quad J(q)=\sum_tj_tq^t.\]

With this notation at hand, we can rewrite the bridge identity. Recall $\Rayleigh_{ef}^b:=Z_{ef}(B_e)Z_{ef}(B_f)
-Z_{ef}(B_e\cap B_f)Z_{ef}=P_e(q)P_f(q)-B(q)A(q)$. Similarly, we have $Z_{ef}(J_{ef})=J(q)$. Hence, we have
\[\dfrac{\Rayleigh_{ef}(\RCZ_{K_n})}{(1-q)^2}=P_e(q)P_f(q)-A(q)B(q)+\dfrac{J(q)A(q)}{1-q}.\]

If we expand all the terms, we can write the equality in a more succinct way, 
\[\dfrac{\Rayleigh_{ef}(\RCZ_{K_n})}{(1-q)^2}=\sum_{k\geq 2}c_kq^k,\]
where $c_k=\mathcal{P}_k+\mathcal{J}_k-\mathcal{B}_k$ can be expressed as a sum of coefficients of the Cauchy products of the terms above,
\[\mathcal{P}_k=\sum_{s+t=k}p_s^ep_t^f, \quad \mathcal{J}_k=\sum_{s+t\leq k}j_sa_t, \quad \mathcal{B}_k=\sum_{s+t=k}b_ta_s.\] 

The main idea of the proof and the following lemmas is to compare the coefficients that we arrived at using the Cauchy product and the bridge identity. While comparing the coefficients, we will use some geometric moves along the edges of the configurations in order to establish several injections. Figure \ref{fig:bridge-injections} gives the four moves we mainly employ. 

\begin{figure}[htbp]
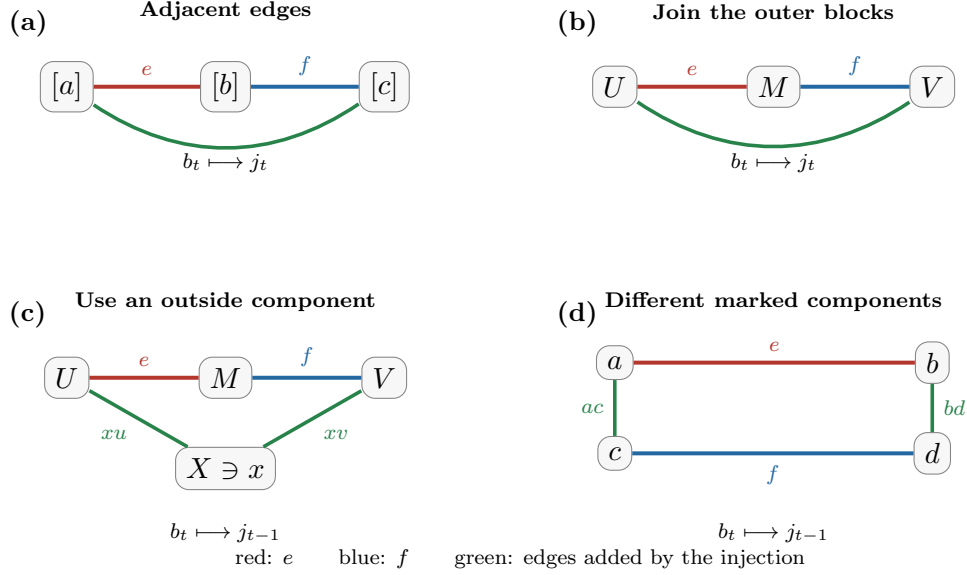

\centering
\InjectionFigure
\caption{The bridge-to-joint-bridge injections.  Parts (a) and (b) add an edge between the two outer blocks and preserve the component count.  Part (c) attaches the outer blocks to a canonically chosen outside component, while part (d) joins two different marked components crosswise; these two moves lower the component count by one.  In every case the green edges can be recovered from the image.}
\label{fig:bridge-injections}
\end{figure}

We begin with the following lemma which compares the bridge configurations to joint bridge configurations when $e$ and $f$ are adjacent edges. We will abuse the notation of $a_t,b_t,j_t$ a little bit. These are polynomials consisting of monomials coming from configurations. If a configuration $A$ is counted in one of them, we will say it belongs to it, even though technically, it belongs to the set the polynomial counts. 

\begin{lemma}
\label{lem:adjacent-injection}
Let $e=ab$, and $f=bc$ be two adjacent edges in $K_n$. We have $b_t\leq j_t$ for every $t\geq 1$. 
\end{lemma}

\begin{proof}
 Suppose $A$ is a configuration that is in $b_t$, so that both $e$ and $f$ are bridges. Note that the edge $ac$ must be absent. Then the map $A \to A+ac$ is a map from $b_t$ into $j_t$ since with $ac$, the edges cease to be lone bridges and become joint bridges. We also note that the number of components does not change since $a$ and $c$ were already in the same component. Now, consider a configuration $A$ in $j_t$ which contains $e=ab,f=bc,ac$. The map $A\to A-ac$ counts a unique configuration from $b_t$. The image has weight $\lambda$ times the weight of $A$, and since $\lambda\geq 1$, we have $b_t\leq j_t$. 
\end{proof}

For the disjoint edges $e=ab$ and $f=cd$, we separate the proof into two parts, the connected part with $t=1$, and disconnected parts with $t\geq 2$. In the connected case, we have some small issues that do not arise otherwise. We first start with $t\geq 2$ which is again an injective argument similar to the one above. 

\begin{lemma}
\label{lem:disjoint-injections}
Let $e=ab$ and $f=cd$ be two disjoint edges in $K_n$. For every $t\geq 2$, we have $b_t\leq j_t+j_{t-1}$. 
\end{lemma}

\begin{proof}
We divide the proof into two cases. Let $A$ be a configuration from $b_t$. Suppose first that the edges $e$ and $f$ are in the same component, say $C\subseteq A$. Since each one is a bridge, $A-e-f$ partitions $C$ into three parts $U,M,V$ with the middle component $M$ containing one of the endpoints of $e$ and $f$. Suppose $a\in U$. Then either $\{b,c\}$ or $\{b,d\}$ is in $M$. We add the edge that is joining the endpoints of $e$ and $f$ that are contained in $U$ and $V$, i.e. if $\{b,c\}\in M$, we add the edge $ad$. Note that this edge was absent, since the endpoints lie in distinct components. If $\kappa(A)=t$, then we observe that $\kappa(A-e-f+ad)=t-1$. Now, consider the map $A\to A+ad$. We observe that this is a map from $b_t$ to $j_t$ by the above considerations. We claim that the inverse of this map is also unique under the condition that $a\in U$. Consider $A+ad$. Then we know that $d\in V$ and $\{b,c\}\in M$. Taking out $ad,e,f$, we can recover the components $U,V,M$ and hence recover $A$ uniquely.  

Next, we look at the case when $\{a,c\}$ or $\{a,d\}$ is in $M$. We have to treat this case separately because performing the same injection as above would annihilate the inverse operation. $ad$ might be an edge in $M$; then we do not know if the graph came from a configuration where $a\in U,d\in V$ or from a configuration where $a,d\in M$. Here, we use the fact that $t\geq 2$. Since $e,f\in C$, we have at least one other component outside $C$. We choose the component $C'$ whose least labeled vertex $x$ is the smallest. Let $u,v$ be the vertices in $U$ and $V$ respectively. Then, we consider the edges $xu,xv$ and add them to the configuration. We had $\kappa(A)=t$, so we have $\kappa(A-e-f+xu+xv)=t-1$. Thus, $A \to A+\{xu,xv\}$ produces a map from $b_t$ to $j_{t-1}$.

This map is injective within each middle pair class. We look at the component not containing $a$; then the vertex $x$ is the unique common neighbour of $u$ and $v$. If we delete $e,f$, the image would recover $x$ and thus the components $U,M,V$. Thus, in the case where $e,f$ are in the same component, we have $b_t\leq j_t+j_{t-1}$ since we always add at least one more edge which multiplies the polynomial by $\lambda\geq 1$. 

Now, let us consider the case where $e,f$ belong to distinct components, say $C_1$ and $C_2$ in the configuration $A$. In this case, note that the edges $ac$ and $bd$ are absent. We consider $A+ac+bd$. Then, $\kappa(A+ac+bd)=t-1$ since the components $C_1$ and $C_2$ merge together. We also note that $e,f$ become joint bridges. Deleting $e$ leaves the path $a,c,d,b$ and deleting $f$ leaves the path $c,a,b,d$, yet deleting $e,f$ together separates the $\{a,c\}$ component from the $\{b,c\}$ component. For the inverse, if we delete the edges $ac,bd$ we recover the source. We also claim that the configurations which contain $ac,bd$ cannot occur from the operations we did above when $e,f$ were in the same component. This is due to the fact that both $ac$ and $bd$ are in the image, yet from before we omitted the pair coming from the parts $U,V$. 

Hence, the image sets are pairwise disjoint and the forward maps are all one-to-one. This proves that $b_t\leq j_t+j_{t-1}$. 
\end{proof}

We will prove the inequality for the case $t=1$ now. A straightforward injection argument is not working here; hence before moving forward with it, we prove a lemma comparing the weighted number of components where an edge is a bridge with the weighted number of components where it is not a bridge. 

\begin{lemma}
\label{lem:bridge-half}
Let $n\geq 3$, and consider $K_n$ with edge weights all equal to $\lambda\geq 1$. Among the connected spanning subgraphs, let $\beta_n$ denote the weighted sum over the configurations in which the edge $xy$ is present and is a bridge, 
\[\beta_n=\sum_{\substack{A \; \text{is connected}\\ A\ni e \; \text{is a bridge}}}\lambda^{|E(A)|}.\]Let $\gamma_n$ denote the weighted sum over the configurations in which either $xy$ is absent or is present but is not a bridge, 
\[\gamma_n=\sum_{\substack{A \; \text{is connected}\\ e\not \in A}}\lambda^{|E(A)|}+\sum_{\substack{A \; \text{is connected}\\ A\ni e \; \text{is not a bridge}}}\lambda^{|E(A)|}.\] Then, we have $\beta_n\leq \gamma_n$.  
\end{lemma}

\begin{proof}
Let $A\in \beta_n$ be a configuration where $xy$ is present and is a bridge. Consider $A-xy$ and its connected components $x\in C_1$ and $y\in C_2$. First, we suppose $|C_2|\geq 2$. Consider $\Gamma(y)\subseteq C_2$, the neighbours of $y$. We construct the map \[A\to (A-xy)\cup \{xv: v\in \Gamma(y)\}.\]

The graph $(A-xy)\cup \{xv: v\in \Gamma(y)\}$ is connected and $xy$ is absent. Thus, we have arrived at a new configuration which belongs to $\gamma_n$. We deleted one edge and added $|\Gamma(y)|\geq 1$ edges; hence the weights increase by $\lambda^{|\Gamma(y)|-1}\geq 1$. The map is invertible on its image. Suppose we have a graph where $xy$ is not present and all of its neighbours send an edge to $x$. We can delete all these edges and add $xy$ to get the unique configuration in $\beta_n$. 

Now, suppose $\{y\}=C_2$. Let $z$ be the least vertex different from $x,y$, and consider the map $A\to A+zy$.  The edge $xy$ now lies on a cycle, and $A$ is recovered by deleting $yz$; hence this is an injection into the class in which $xy$ is present but not a bridge.  Its weight is $\lambda$ times the source weight.  The two image classes are disjoint because $xy$ is absent in the first and present in the second.  This proves the weighted inequality. 
\end{proof}

Now, consider $\mathcal{C}_n=\beta_n+\gamma_n$, which is the sum over all configurations which are connected: 
\[\mathcal{C}_n=\sum_{A\; \text{is connected}}\lambda^{|E(A)|}.\]
If we decompose at the bridge $xy$, we get the following
\[\beta_n=\lambda\sum_{r=1}^{n-1}\binom{n-2}{r-1}\mathcal{C}_r\mathcal{C}_{n-r}.\]
We will prove the following quantitative bounds on the configurations containing $xy$ as a bridge, which will give us the proportion among all the connected configurations. 

\begin{lemma}
\label{lem:quantitative-bridge}
For every $\lambda \geq 1$, and $n\geq 5$, we have
\[\beta_{n+1}\leq 3\lambda \mathcal{C}_n, \; \text{and} \quad 4\beta_n\leq \lambda \mathcal{C}_n.\]
\end{lemma}

\begin{proof}
Let $N=\binom{n}{2}$ be the number of edges in $K_n$ with labeled vertices, and set $\mu:=1+\lambda$. The sum of all configurations with weights $\lambda$ becomes $\mu^N$ which can be seen by expanding $(1+\lambda)^N$. We note that the second inequality follows from the first. Let $A$ be a connected configuration in $K_{n-1}$. Joining the new vertex to any non-empty subset of the vertices of $A$ produces a connected configuration in $K_n$. Different choices produce distinct configurations, so we have \[\mathcal{C}_n\geq (\mu^{n-1}-1)\mathcal{C}_{n-1}.\] 
Using the first inequality, we have for $n\geq 6$,
\[4\beta_n\leq 12\lambda \mathcal{C}_{n-1} \leq \dfrac{12\lambda}{(1+\lambda)^{n-1}-1}\mathcal{C}_n\leq \dfrac{12}{31}\lambda \mathcal{C}_n\leq \lambda \mathcal{C}_n.\]

Let us focus on the first inequality. We can classify the disconnected graphs according to the connected component vertex $1$ is in. On the one hand, we have $\mu^N-\mathcal{C}_n$ many disconnected components. On the other hand, we have \[\sum_{r=1}^{n-1}\binom{n-1}{r-1}\mathcal{C}_r\mu^{\binom{n-r}{2}},\]
many disconnected components. In the following estimates within this proof, $m$ is another notation for $n$, the number of vertices. Here, we count over $S$ where $1\in S$, and $|S|=r$. $\binom{m-1}{r-1}$ counts the number of vertices that can be in $S$, and $\mathcal{C}_r$ counts the number of connected components on $S$. For the rest, we can take any configuration. 

If we arrange the equality, and use the fact that $\mathcal{C}_r\leq \mu^{\binom{r}{2}}$, we get
\[1-U_n\leq \dfrac{\mathcal{C}_n}{\mu^N},\]
with $U_n=\sum_{r=1}^{n-1}\binom{n-1}{r-1}\mu^{-r(m-r)}$. Here, the terms $r=1$ and $n-1$ contribute at most $n2^{1-n}$, and the remaining terms contribute at most $2^{n-1}2^{-2(n-2)}=2^{3-n}$. We can write \[\beta_{n+1}=\lambda\sum_{r=1}^{n}\binom{n-1}{r-1}\mathcal{C}_r\mathcal{C}_{n+1-r}=2\lambda \mathcal{C}_n+\lambda\sum_{r=2}^{n-1}\binom{n-1}{r-1}\mathcal{C}_r\mathcal{C}_{n+1-r}=:2\lambda \mathcal{C}_n+\lambda R_n.\]
We set $x=r-1$ and bound $R_n$ as follows:
\[\dfrac{R_n}{\mu^N}\leq \sum_{x=1}^{n-2}\binom{n-1}{x}\mu^{-x(n-1-x)}\leq (n+3)2^{3-m}.\]
Together with the first inequality, we get 
$R_n\leq \mathcal{C}_n$ since \[(n+3)2^{3-m}+n2^{1-m}+2^{3-m}\leq 1,\]
for $n\geq 7$. 

We handle the boundary cases $n=5$ and $n=6$, and the boundary case $n=5$ for the second inequality by expanding the recursion $\mathcal{C}_n=\mu^{N}-\sum_{r=1}^{n-1}\binom{n-1}{r-1}\mathcal{C}_r\mu^{\binom{n-r}{2}}$. A direct expansion with $u=\lambda-1$ gives the following nonnegative expressions: 
\begin{align*}
3\lambda \mathcal{C}_5-\beta_6
&=\lambda^5(328+940u+872u^2+412u^3+110u^4+16u^5+u^6),\\
3\lambda \mathcal{C}_6-\beta_7
&=\lambda^6(16384+65240u+106570u^2+101940u^3+64660u^4\\
&\hspace{10mm}+28604u^5+8950u^6+1955u^7+285u^8+25u^9+u^{10})\\
\lambda \mathcal{C}_5-4\beta_5&=(1+u)^4(328+1692u+2106u^2+1362u^3+530u^4+126u^5+17u^6+u^7).
\end{align*}
\end{proof}

With these at hand, we are now ready to prove the $t=1$ case.
\begin{lemma}
\label{lem:connected-comparison}
Let $e=ab,f=cd \in E(K_n)$ be two disjoint edges and let $b_1,j_1$ be defined as before. For $\lambda\geq 1$ and $n\geq 7$, we have $b_1\leq j_1$. 
\end{lemma}

\begin{proof}
We delete $e,f$ from the configurations and compare the two sides. Let $A\in b_1$ be a connected subgraph with bridges $e,f\in A$. In $A-e-f$, we get three distinct components $U,M,V$ in a path, the middle $M$ containing one endpoint of each of $e$ and $f$, and two outer parts $U,V$ containing the other endpoints. Let $B\in j_1$ be a connected subgraph where $e,f\in B$ are joint bridges. In $B-e-f$, we have two components $C_1,C_2$ each containing two vertices from the endpoints of $e,f$. We either have $ac|bd$, which means $\{a,c\}\in C_1$ and $\{b,d\}\in C_2$, or $ad|bc$. We fix $ac|bd$ and the components $C_1,C_2$ as above with sizes $|C_1|=r$ and $|C_2|=s$.

We consider the subgraphs of $C_1$ where there is a 3-path decomposition $a\in U, \{b,d\}\in M$, and $c\in V$. Adding the edge $ac$ to such configurations takes us into the decomposition we have for $j_1$ with $ac|bd$. Since we add $ac$, we get an extra weight $\lambda$. This gives us the following bijection: 
\[\left\{\begin{array}{c}\text{spanning subgraphs on} \; U \; \text{with two}\\ \text{connected components separating}\; a,c \end{array} \right\} \xleftrightarrow[\text{delete} \; ac]{\text{add} \; ac} \left\{\begin{array}{c}\text{connected spanning subgraphs on} \; U \;\\ \text{in which}\; ac \; \text{is a bridge} \end{array} \right\}.\]

The two-component graphs on $C_1$ have total weight $\frac{\beta_r}{\lambda}$, and the component $C_2$ is free to be chosen independently under the condition that it is connected. Thus, the total weight of such configurations is $\frac{\beta_r \mathcal{C}_s}{\lambda}$. In a similar manner, considering the class with $bd$ instead of $ac$ gives rise to total weight $\frac{\beta_s\mathcal{C}_r}{\lambda}$ after adding the edge $bd$. 

When we consider the class $j_1$, we have the total weight $\mathcal{C}_r\mathcal{C}_s$; both sides are connected with sizes $r$ and $s$ respectively. Hence, for these fixed vertex sets, the polynomial identity corresponding to the difference $b_1-j_1$ can be written as
\[\psi_{r,s}:=\dfrac{\beta_r\mathcal{C}_s+\mathcal{C}_r\beta_s}{\lambda}-\mathcal{C}_r\mathcal{C}_s.\]

We sum over all configuration sizes. We need to choose $r-2$ vertices for the set containing $ac$, and similarly for the set containing $bd$; we get
\[b_1-j_1=2\sum_{r=2}^{n-2}\binom{n-4}{r-2}\psi_{r,n-r}.\]

For $r,s\geq 3$, Lemma \ref{lem:bridge-half} gives \[\dfrac{\psi_{r,s}}{\mathcal{C}_r\mathcal{C}_s}=\dfrac{1}{\lambda}\left(\dfrac{\beta_r}{\mathcal{C}_r}+\dfrac{\beta_s}{\mathcal{C}_s}\right)-1\leq 0.\]
For blocks of size $2$, we have $\beta_2=\mathcal{C}_2=\lambda$, and hence
\[\psi_{2,s}=\beta_s-(\lambda-1)\mathcal{C}_s.\]
When $\lambda\geq 2$, $\psi_{2,s}\leq 0$, since we have $\beta_s\leq \mathcal{C}_s$.

Now, we look at weights $1\leq \lambda \leq 2$. Let $m=n-3$, and consider the boundaries $(m-1)\psi_{2,n-2}$, and the terms with $r=3,s=m$. Every other term is already non-positive by Lemma \ref{lem:bridge-half}. We have $\beta_3=2\lambda^2$ and $\mathcal{C}_3=\lambda^2(\lambda+3)$; thus using Lemma \ref{lem:quantitative-bridge} gives
\begin{align*}
\psi_{3,m}&=2\lambda \mathcal{C}_m+\dfrac{\mathcal{C}_3}{\lambda}\beta_m-\mathcal{C}_3\mathcal{C}_m\leq 2\lambda \mathcal{C}_m+\dfrac{1}{4}\mathcal{C}_3\mathcal{C}_m-\mathcal{C}_3\mathcal{C}_m\\
&=-\lambda\left(\dfrac{3}{4}\lambda(\lambda+3)-2 \right)\mathcal{C}_m \leq -\lambda \mathcal{C}_m.
\end{align*}

For the $2$ boundaries, we have $\psi_{2,m+1}\leq \beta_{m+1}\leq 3\lambda \mathcal{C}_m$. Thus, we have for $n\geq 8$,
\begin{align*}
\dfrac{b_1-j_1}{2}&=\sum_{r=2}^{n-2}\binom{n-4}{r-2}\psi_{r,n-r}\leq 2\psi_{2,m+1}+2(m-1)\psi_{3,m}\\
&\leq 6\lambda \mathcal{C}_m+2(m-1)(-\lambda \mathcal{C}_m)=2\lambda \mathcal{C}_m(4-m)\leq 0. 
\end{align*}

For $n=7$, we directly calculate the values to get
\[b_1-j_1=-4\lambda^4(8+1358u+1822u^2+1179u^3+462u^4+113u^5+16u^6+u^7),\]
with $u=\lambda-1$ which is non-positive.  
\end{proof}

With lemmas \ref{lem:adjacent-injection}, \ref{lem:disjoint-injections}, and \ref{lem:connected-comparison}, we will prove coefficientwise positivity for the Rayleigh difference. 

\begin{thm}
\label{thm:coefficientwise}
Suppose we have equal edge weights $\lambda\geq 1$ on the complete graph $K_n$, and let $e,f$ be two distinct edges. Let $c_k$ be the coefficients we defined above. For every $k\geq 2$, $c_k\geq 0$. 
\end{thm}

\begin{proof}
We recall that $c_k=\mathcal{P}_k+\mathcal{J}_k-\mathcal{B}_k$ with
\[\mathcal{P}_k=\sum_{s+t=k}p_s^ep_t^f, \quad \mathcal{J}_k=\sum_{s+t\leq k}j_sa_t, \quad \mathcal{B}_k=\sum_{s+t=k}b_ta_s.\]

For adjacent edges, Lemma \ref{lem:adjacent-injection} assigns the term $b_ta_{k-t}$ of $\mathcal{B}_k$ to the distinct term $j_ta_{k-t}$ of $\mathcal{J}_k$. Thus, we get the positivity. 

Suppose the edges are disjoint and $n\geq 7$. Lemma \ref{lem:connected-comparison} treats the case when $t=1$ by establishing the inequality $b_1\leq j_1$ for $n\geq 7$. For $t\geq 2$, Lemma \ref{lem:disjoint-injections} assigns $b_ta_{k-t}$ of $\mathcal{B}_k$ to $j_ta_{k-t}+j_{t-1}a_{k-t}$ of $\mathcal{J}_k$. 

We note that as $t$ varies, these occupy distinct indices: the first one occupies the diagonal $s+t=k$, and the second one occupies the diagonal $s+t=k-1$. Hence, $\mathcal{B}_k\leq \mathcal{J}_k$, which proves $c_k\geq 0$. 

All that remains are the cases $K_4,K_5,K_6$. This is a finite check which holds for all $\lambda\geq 1$. 
\end{proof}

\begin{proof}[Proof of Theorem \ref{thm:complete-graph}.]
For $0<q<1$, the series $\sum_{k\geq2}c_kq^k$ converges since its coefficients are eventually constant. Theorem~\ref{thm:coefficientwise} therefore gives
\[
\Rayleigh_{ef}(\RCZ_{K_n})=(1-q)^2\sum_{k\geq2}c_kq^k\geq 0.\qedhere
\]
\end{proof}

\begin{proof}[Proof of Corollary \ref{cor:connected-subgraphs}.]
We fix $\lambda\geq 1$ to be the common edge weight. Letting $q\to 0$ with fixed $\lambda$ conditions the random-cluster model on the minimum $\kappa$ parameter, which is one. Hence, the measure becomes concentrated on the connected spanning subgraphs of $K_n$ weighted with $\lambda^{|A|}$. Since we have the positivity result for $q>0$, the limit by continuity preserves the inequality. Thus, this result is equivalent to having $c_2\geq 0$ which we proved above. 
\end{proof}

We have proved our results under the condition that the weights are all equal and are greater than $1$. The most natural extension is to ask what happens when the weights are arbitrary and possibly less than $1$. We prove a special case of this: if the edges among the endpoints of $e,f$ have weight $\geq 2$, then we are able to achieve more freedom compared to the previous settings.

We reuse $a_t,p_t^e,p_t^f,b_t,j_t$ and $c_k$ for nonuniform edge weights: in each configuration sum, replace $\lambda^{|A\setminus\{e,f\}|}$ by $\prod_{g\in A\setminus\{e,f\}}z_g$, and define $c_k$ by the same coefficient formulas as before.

\begin{prop}
\label{prop:weighted-cross-edge}
Let $e=ab,f=cd\in K_n$ be distinct edges, and $0<q<1$. Suppose $z_{ac},z_{ad},z_{bc},z_{bd}\geq 2$. Then, $b_1\leq j_1$ with all other edge weights arbitrary and nonnegative. If, in addition, every edge $g$ with $|g\cap \{a,b,c,d\}|=1$ has weight $\geq 1$, $c_k\geq 0$ for all $k$, and thus $e$ and $f$ are negatively correlated. For adjacent edges $e=ab,f=cd$, if $z_{ac}\geq 1$, then $e,f$ are negatively correlated.
\end{prop}

\begin{proof}
We set $z_e=z_f=1$, since the Rayleigh difference is independent of these two weights. Let $A\in b_1$, and consider the graph $A-e-f$, which has $3$ components, $U,M,V$. Let $M$ be the middle component which contains the two vertices from the endpoints of $e$ and $f$. Let $a\in U$ and $c\in V$ without loss of generality. Consider the map \[\psi: A \to A+ac.\]
We have $\psi(A)=A+ac \in j_1$, and the joint bridges $e,f$, when taken out, leave two components $C_1,C_2$ with $ac$ on one part and $bd$ on the other part. We also observe that the weight of the configuration $A+ac$ is the same as the weight of the configuration $A$ multiplied by $z_{ac}\geq 2$. 

Now, let us look at the preimage $\psi^{-1}(B)$ of a configuration $B\in j_1$, which will be a set in $b_1$. We claim that the size of the preimage will be at most $2$. Consider $B-e-f$. Suppose without loss of generality that we have $a,c\in C_1$ and $b,c\in C_2$. Then, the only way $B$ is the image of a configuration $A$ under $\psi$ is to have $ac\in C_1$ as a bridge or to have $bd\in C_2$ as a bridge. Then, we would have $\psi(B-ac)=B$, or $\psi(B-bd)=B$. Thus, there are at most $2$ such configurations, and the weights of the configurations are the weight of $B$ multiplied by $\frac{1}{z_{ac}}$ or $\frac{1}{z_{bd}}$ respectively. Letting $w(A)$ denote the weight of $A$, we have
\[\sum_{A\in \psi^{-1}(B)}w(A)\leq w(B)\left(\dfrac{1}{z_{ac}}+\dfrac{1}{z_{bd}} \right)\leq w(B).\]

If we sum over $B$, we get $b_1\leq j_1$ since the left side will count $b_1$ whereas the right will be $j_1$. 

The injections in Lemma~\ref{lem:disjoint-injections} only add edges incident with one of $a,b,c,d$ except $e,f$. Under the additional lower bound $1$, the weight of an image is at least the weight of its source, so the proof of Theorem~\ref{thm:coefficientwise} applies coefficient by coefficient.  In the adjacent case, Lemma~\ref{lem:adjacent-injection} adds only $ac$, giving the last assertion.
\end{proof}

\section{Connectivity Correlation and Gibbs Measures}
\label{sec:connectivity-programme}
Let $G=(V,E)$ be a graph, and let $e=ab,f=cd\in E$ be fixed edges. We set $H=G-e-f$. We recall the event $N_e$. Consider a configuration $A\subseteq E(H)$, and let $C_A(x)$ be the component of the vertex $x$ in $A$.  We defined the event $N_e$ as $N_e=\{C_A(a)=C_A(b)\}$ in Section \ref{sec: preliminarylemmas}. We can write this event as the connectivity event $\{a\leftrightarrow b\}$, meaning that the vertices $a,b$ are connected. In this section, we will examine the lower inequalities \begin{align*}
0&\leq \Cov_H^{\mathcal{F}}(\mathbf{1}_{N_e},\mathbf{1}_{N_f}).
\end{align*}

We also defined \[\Gamma_H(a,b;c,d):=Z_H(N_e\cap N_f)Z_H-Z_H(N_e)Z_H(N_f).\]
With $\Gamma_H(a,b;c,d)$, one can write the covariance as \[\Cov_H^{\mathcal{F}}(\mathbf{1}_{N_e},\mathbf{1}_{N_f})=\dfrac{\Gamma_H^{\mathcal{F}}(a,b;c,d)}{(\ForestZ_H)^2}.\]
 
We prove that considering the two-forests as we have done in Section \ref{sec: arboreal-gas}, we have the non-negativity for the covariance of bridge events. We let $\tau^{\times}_2(G;e,f)$ denote the number of spanning two-forests of $G$ which separate both endpoints $a,b$ and $c,d$. 
\begin{prop}
\label{prop:large-lambda-connectivity}
Let $G=(V,E)$ be a graph, $e,f\in E$ be distinct edges, and let $H=G-e-f$. Assume $H$ is connected. Let $\ForestZ_H(\lambda)$ be the arboreal gas measure with equal fugacity $\lambda$, and denote by $\Cov_H^{\mathcal{F}}$ the covariance under this measure. We have
\[\Cov_H^{\mathcal{F}}(\mathbf{1}_{N_e},\mathbf{1}_{N_f})=\dfrac{\tau_2^{\times}(H;e,f)}{\tau(H)}\lambda^{-1}+O(\lambda^{-2}).\]
Moreover, we either have $\tau_2^{\times}(H;e,f)>0$, in which case the covariance is positive for all sufficiently large $\lambda$, or $N_e$ and $N_f$ depend on different cycles, and their covariance vanishes identically for all $\lambda$. 
\end{prop}

\begin{proof}
We assume $H=G-e-f$ is connected. First, note that complementing the events preserves the covariance; hence we will work with events $D_e:=\{a \not \leftrightarrow b\}$ and $D_f:=\{c \not \leftrightarrow d\}$. We have \[\Cov_H^{\mathcal{F}}(\mathbf{1}_{N_e},\mathbf{1}_{N_f})=\Prob_H^{\mathcal{F}}(D_e\cap D_f)-\Prob_H^{\mathcal{F}}(D_e)\Prob_H^{\mathcal{F}}(D_f).\]
On $D_e\cap D_f$ when both endpoints are disconnected, a spanning forest can have at least $2$ components, and thus at most $|V(H)|-2$ edges. The $2$-component case is exactly counted by $\tau_2^{\times}(H;e,f)$. Thus, we have
\[\ForestZ_H(D_e\cap D_f)=\tau_2^{\times}(H;e,f)\lambda^{|V(H)|-2}+O(\lambda^{|V(H)|-3}).\]

We also have $\ForestZ_H(\lambda)=\tau(H)\lambda^{|V(H)|-1}+O(\lambda^{|V(H)|-2}),$ which gives us the probability 
\[\Prob_H^{\mathcal{F}}(D_e\cap D_f)=\dfrac{\tau_2^\times(H;e,f)}{\tau(H)}\lambda^{-1}+O(\lambda^{-2}).\]

On $D_e$ and $D_f$ separately, the forests can have at most two connected components as well. The same analysis then gives 
$\Prob_{H}^{\mathcal{F}}(D_e)=O(\lambda^{-1})$, and $\Prob_H^{\mathcal{F}}(D_f)=O(\lambda^{-1})$.

Next, we show that whenever $\tau_2^\times(H;e,f)>0$, the edges $e$ and $f$ lie on the same cycle in $G$. Suppose a spanning two-forest of $H$ separates both
endpoint pairs of $e$ and $f$. Adding $e$ joins its two components and produces a spanning tree. Adding $f$ then creates a unique cycle, which also contains $e$. Thus $e$ and $f$ lie on
a common cycle of $G$.

Conversely, suppose a cycle of $G$ contains both $e$ and $f$.
Deleting these edges from the cycle leaves two vertex-disjoint
paths in $H$, possibly with one path consisting of a single
vertex. The paths separate both endpoint pairs.
Since $H$ is connected, their union extends to a spanning
tree $T$ of $H$. Contract the two paths in $T$ and delete
an edge on the path between the resulting two vertices.
In the original tree, this deletion produces a spanning
two-forest containing both original paths and separating
both marked pairs. Hence $\tau_2^\times(H;e,f)>0$.

Finally, suppose no cycle of $G$ contains both $e$ and $f$.
Then $e$ and $f$ belong to different components of the
graphic matroid of $G$, say $M_e$ and $M_f$.
The forest condition and the weights $\lambda^{|A|}$
factor over the matroid components. Consequently, the
restrictions of a random forest of $H$ to these components
are independent. Furthermore, $N_e$ depends only on the edges in $M_e$: any $a$--$b$ path in $H$, together with $e$, forms a cycle
of $G$, so every edge of that path belongs to $M_e$.
Likewise, $N_f$ depends only on the edges in $M_f$.
The two events are therefore independent, and their
covariance vanishes for every $\lambda>0$. Figure \ref{fig:cycle-two-forest} shows the distinct cases we analyzed above.
\end{proof}

\begin{figure}[htbp]
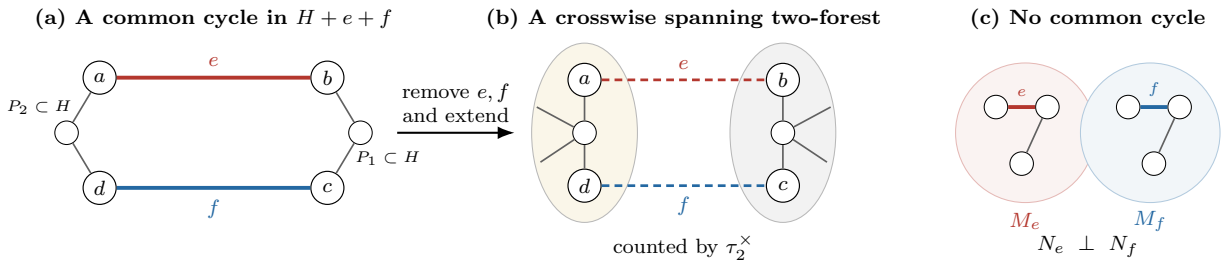

\centering
\resizebox{.98\textwidth}{!}{\CycleTwoForestFigure}
\caption{The large-fugacity dichotomy.  Removing $e$ and $f$ from a common cycle leaves two crosswise paths in $H$, which extend to a spanning two-forest counted by $\tau_2^\times$.  If no common cycle exists, $e$ and $f$ lie in different graphic-matroid components, and the connectivity events $N_e,N_f$ are independent factors of the arboreal-gas measure.}
\label{fig:cycle-two-forest}
\end{figure}

Proposition \ref{prop:large-lambda-connectivity} reveals the positive correlation of connectivity events $\{a\leftrightarrow b\}$ and $\{c\leftrightarrow d\}$ for large $\lambda$ fugacity. We conjecture that this holds for all $\lambda$ and also for the random-cluster measure for any $q>0$. 

\begin{conj}
\label{conj: forest-connectivity}
Let $G=(V,E)$ be a finite graph, and let $a,b,c,d\in V$ be distinct vertices. Under the arboreal gas measure with positive fugacity at each edge, we have
\[\Gamma_G^{\mathcal{F}}(a,b;c,d) \geq 0.\]
In other words, 
\[\Prob_G^{\mathcal{F}}(a\leftrightarrow b, c \leftrightarrow d)\geq \Prob_G^{\mathcal{F}}(a\leftrightarrow b)\Prob_G^{\mathcal{F}}(c \leftrightarrow d).\]
\end{conj}

\begin{conj}
\label{conj: rc-connectivity}
Let $G=(V,E)$ be a finite graph, and let $a,b,c,d\in V$ be distinct vertices. For every $0<q<1$, under the random-cluster measure, we have
\[\Gamma_G^{\mathrm{RC}}(a,b;c,d) \geq 0.\]
In other words, 
\[\Prob_G^{\mathrm{RC}}(a\leftrightarrow b, c \leftrightarrow d)\geq \Prob_G^{\mathrm{RC}}(a\leftrightarrow b)\Prob_G^{\mathrm{RC}}(c \leftrightarrow d).\]
\end{conj}

We note that Conjecture \ref{conj: rc-connectivity} is true for $q\geq 1$. One can prove this assertion by noting that the connectivity events are non-decreasing events under addition of edges, and using the FKG inequality. Hence, if Conjecture \ref{conj: rc-connectivity} is true, then the connectivity events are positively correlated for all $q>0$. 

For the arboreal gas measure, we can strengthen the conjecture at different levels: 
\begin{enumerate}
    \item Every monomial coefficient of the multivariate polynomial $\Gamma_G^{\mathcal{F}}(a,b;c,d;\mathbf{z})$ is nonnegative. 
    \item After setting all $z_g=\lambda$, every coefficient of $\Gamma_G^{\mathcal{F}}(a,b;c,d)$ is nonnegative. 
    \item $\Gamma_G^{\mathcal{F}}(a,b;c,d;\lambda)\geq 0$ for every $\lambda>0$.
\end{enumerate}

We note that each item implies the next. Numerically, there is strong evidence that $1.$ actually holds. Let us look at the smallest possible example. 

\begin{example}
Consider $K_4$ and the edges $e=12, f=34$. Let $H=K_4-e-f$. If we expand directly, we have 
\begin{align*}
\Rayleigh_{12,34}(\ForestZ_{K_4})=&z_{13}z_{24}(1+z_{13})(1+z_{24})
   +z_{14}z_{23}(1+z_{14})(1+z_{23})-2z_{13}z_{14}z_{23}z_{24} \\
=&(z_{13}z_{24}-z_{14}z_{23})^2+z_{13}z_{24}(1+z_{13}+z_{24})+z_{14}z_{23}(1+z_{14}+z_{23}).
\end{align*}
We can see that we have pointwise non-negativity after writing it in square form; however, coefficientwise, in the Rayleigh difference we have $-2$ as the coefficient of $z_{13}z_{14}z_{23}z_{24}$. Hence, we do not have coefficientwise positivity in the Rayleigh difference. However, if we consider the connectivity correlation, we have 
\begin{align*}
\Gamma_H^{\mathcal{F}}(1,2;3,4)
=&z_{14}z_{23}(z_{13}+z_{24})(1+z_{14})(1+z_{23})\\
&+z_{13}z_{24}(z_{14}+z_{23})(1+z_{13})(1+z_{24})\\
&+z_{13}z_{14}z_{23}z_{24}
  (4+z_{13}+z_{14}+z_{23}+z_{24}).
\end{align*}
We observe that every coefficient appearing above is positive. Hence, we have non-negativity coefficientwise in $\Gamma_H^{\mathcal{F}}(1,2;3,4)$.
\end{example}

We note that a similar example can be calculated for the random-cluster measure and we again have coefficientwise positivity in the connectivity correlation. 

The lower covariance bound has two further applications in statistical physics. Positive correlation between connectivity events implies super-multiplicativity of the connection probability function\footnote{This is sometimes called the two-point function in percolation theory.} and hence the existence of directional inverse correlation lengths. We then show that extending the two-sided covariance bound to an arbitrary auxiliary forest yields a quantitative bound on boundary influence and, in two dimensions, uniqueness of the infinite-volume Gibbs measure.

\begin{prop}
\label{prop: super-multiplicavity}
Let $\mu$ be a translation-invariant probability measure on the subgraphs of $\mathbb{Z}^d$. Suppose the connectivity events $\{a\leftrightarrow b\}$, $\{c\leftrightarrow d\}$ are positively correlated for any $a,b,c,d\in V(\mathbb{Z}^d)$. Let \[\theta(x):=\mu(0 \leftrightarrow x)\]
be the connection probability (two-point) function, and assume $\theta(x)>0$ for any $x\in \mathbb{Z}^d$. Then, 
\[\theta(x+y)\geq \theta(x)\theta(y), \quad \forall x,y\in \mathbb{Z}^d.\]
Moreover, for every $x\in \mathbb{Z}^d$, the inverse correlation length defined as \[\xi(x):=\operatorname{lim}_{n\to \infty}-\dfrac{1}{n}\operatorname{log}(\theta(nx))\]
exists in $[0,\infty]$, and satisfies 
\[\xi(x+y)\leq \xi(x)+\xi(y), \quad \forall x,y \in \mathbb{Z}^d.\]
When $\xi$ is finite, it extends by continuity and homogeneity to a convex seminorm on $\mathbb{R}^d$; in a phase with exponential decay in every direction, it is a norm.
\end{prop}

\begin{proof}
We have $\{0 \leftrightarrow x+y\} \supseteq \{0\leftrightarrow x\}\cap \{x \leftrightarrow x+y\}$. Using translation invariance and positive correlation of connectivity events, we get
\begin{align*}
\theta(x+y)&=\mu(0\leftrightarrow x+y)\geq \mu(0\leftrightarrow x)\mu(x\leftrightarrow x+y) \\
&=\mu(0\leftrightarrow x)\mu(0\leftrightarrow y)=\theta(x)\theta(y).
\end{align*}

Let $a_n=-\operatorname{log}(\theta(nx))$. We have $a_{n+m}\leq a_n+a_m$; hence $\operatorname{lim}_{n\to \infty}\frac{a_n}{n}=\operatorname{inf}_{n\geq 1}\frac{a_n}{n}$ by Fekete's lemma. Thus, we have $0\leq \xi(x)\leq -\operatorname{log}(\theta(x))<\infty$. We then have
\[-\dfrac{1}{n}\operatorname{log}(\theta(n(x+y)))\leq -\dfrac{1}{n}\operatorname{log}(\theta(nx))-\dfrac{1}{n}\operatorname{log}(\theta(ny)),\]
and taking the limit proves $\xi(x+y)\leq \xi(x)+\xi(y)$. 

By translation invariance and symmetry of the connectivity function, we observe that $\theta(x)=\theta(-x)$, and hence $\xi(-x)=\xi(x)$. Moreover, observe that 
\[\xi(kx)=\operatorname{lim}_{n\to \infty}-\dfrac{\operatorname{log}(\theta(nkx))}{n}=k\operatorname{lim}_{n\to \infty}-\dfrac{\operatorname{log}(\theta(nkx))}{nk}=k\xi(x).\]
Using these, and the fact that $\xi(0)=0$, we can extend the function to $\mathbb{Q}^d$, and using continuity, we can extend it to $\mathbb{R}^d$. Next, note that for $x,y\in \mathbb{Q}^d$, 
\[|\xi(x)-\xi(y)|\leq \xi(x-y)\leq \operatorname{max}_{1\leq i \leq d}\xi(e_i)\|x-y\|_1,\]
where $e_i$ are the standard basis vectors. Hence, the extension is uniquely defined. 
For the last claim, suppose $\theta(x)\leq \operatorname{exp}(-c\|x\|_2)$. Then, we have $\xi(x)\geq c\|x\|_2$, and this holds on all of $\mathbb{R}^d$ by continuity and homogeneity. Thus, we have $\xi(x)>0$ for all $x\neq 0$. With the other properties we established, this shows that $\xi$ is a norm. We note that the assumption of exponential decay is necessary. Otherwise, one can take $\theta\equiv 1$, which satisfies all the assumptions, and observe that $\xi=0$ everywhere, which is not a norm. 
\end{proof}

Next, we give a sufficient condition for uniqueness of the
infinite-volume Gibbs measure on $\mathbb Z^2$. We extend
the formulation of the covariance bounds from a single edge
to a forest $F$, which represents connections imposed at
the boundary. Assuming both bounds hold in this setting,
we show that boundary influence vanishes as the finite
regions grow, yielding uniqueness.

Let us first introduce the relevant
boundary conditions and events. Let $G\subseteq\mathbb Z^d$
be a finite subgraph, with boundary
\[\partial G=\{v\in V(G):v\text{ is incident to an edge of }
E(\mathbb Z^d)\setminus E(G)\}.\]
A boundary condition is a partition of $\partial G$;
vertices in the same block are regarded as connected
outside $G$. We represent each block by an auxiliary tree
on its vertices and denote the union of these trees by $F$.
The edges of $F$ are distinct from those of $G$ and are
not sampled. If an auxiliary edge has the same endpoints
as an edge of $G$, the two are treated as distinct parallel
edges.

Let $A$ be sampled from the free arboreal-gas measure on $G$.
We say that $A$ is compatible with $F$ if $A\cup F$ is a
forest, and define
\[B_F:=\{A\cup F\text{ is a forest}\}.\]
The arboreal-gas measure with boundary condition represented
by $F$ is
\[\Prob_G^{F,\mathrm{bdry}}(\mathcal E):=\Prob_G^{\mathcal{F}}(\mathcal E\mid B_F).\]
Equivalently, its probability mass function is
\[\Prob_G^{F,\mathrm{bdry}}(A)=\frac{1}{Z_G^{F,\mathrm{bdry}}}\left(\prod_{h\in A}z_h\right)\mathbf1_{B_F},\qquad A\subseteq E(G),\]
where
\[Z_G^{F,\mathrm{bdry}}=\sum_{\substack{A\subseteq E(G)\\
A\cup F\text{ is a forest}}}\prod_{h\in A}z_h.\]
The free boundary condition corresponds to $F=\varnothing$.

Now fix a further auxiliary edge $g=xy$, distinct from
the edges of both $G$ and $F$, and put
\[D_g:=\{A\cup\{g\}\text{ is a forest}\}.\]
Since $A$ is a forest, this is precisely the disconnection
event
\[D_g=\{x\not\leftrightarrow y\text{ in }A\}=N_g^c.\]
We also define
\[J_{F,g}:=(B_F\cap D_g)\setminus\{A\cup F\cup\{g\}\text{ is a forest}\}.\]
Thus $J_{F,g}$ occurs when $F$ and $g$ can each be added
separately, but adding them together creates a cycle.
Equivalently,
\[J_{F,g}=B_F\cap
\{x\not\leftrightarrow y\text{ in }A,\,x\leftrightarrow y\text{ in }A\cup F\}.\]

We consider the following extension of our covariance bounds:
\[
0\leq\Cov_G^{\mathcal{F}}(\mathbf1_{B_F},\mathbf1_{D_g})\leq
\Prob_G^{\mathcal{F}}(J_{F,g}).\]
All probabilities and covariances in this inequality are
taken under the free arboreal-gas measure on $G$.
When $F=\{e\}$ and $g=f$, we have $B_F=N_e^c$ and
$D_g=N_f^c$. Since complementing both events preserves
covariance, the bounds reduce to our bound from Section \ref{sec: preliminarylemmas}
\[0\leq\Cov_G^{\mathcal{F}}(\mathbf1_{N_e},\mathbf1_{N_f})
\leq\Prob_G^{\mathcal{F}}(J_{ef}).\]

For the infinite-volume interpretation, we use the
augmented Gibbs framework of Halberstam and Hutchcroft
\cite[Section~2.1]{halberstam-hutchcroft-2024}.
In this framework, a forest is accompanied by consistent
boundary partitions recording exterior connectivity.
Conditional on the exterior edges and this connectivity
information, its restriction to a finite region has the
corresponding finite-volume arboreal-gas law.

By \cite[Proposition~2.1]{halberstam-hutchcroft-2024},
the infinite-volume Gibbs measures are precisely the
subsequential weak limits of finite-volume measures along
exhaustions\footnote{For any graph $G$, an increasing sequence of finite subgraphs of $G$ whose union is the entire graph $G$ is called an exhaustion \cite{halberstam-hutchcroft-2024}.}, allowing possibly random boundary conditions.
Weak convergence here means convergence of the joint
distribution of every fixed finite collection of edges.
Our aim is to show that all such limits coincide.

To match the boundary convention in that paper without
ambiguity concerning self-loops created by boundary
identifications, we use the exhaustion
\[G_N=\mathbb Z^d[\{x\in\mathbb Z^d:\|x\|_1\leq N\}].\]
Its boundary contains no adjacent vertices, so identifying
vertices in the same boundary block creates no self-loops
from edges of $G_N$. On these regions, the boundary
condition above agrees with the quotient-graph formulation
of \cite{halberstam-hutchcroft-2024}.

We next explain how the two covariance bounds control the
effect of boundary conditions on the inclusion probability
of an edge. Let $g=xy\in E(G)$ have fugacity $z_g$, put
$H=G-g$, and let $F$ be an auxiliary forest representing
connections between boundary vertices of $G$.
All probabilities in the following definitions are taken
under the free arboreal-gas measure on $H$:
\[b=\Prob_H^{\mathcal{F}}(B_F),\qquad d=\Prob_H^{\mathcal{F}}(D_g),\qquad c=\Prob_H^{\mathcal{F}}(B_F\cap D_g),\qquad j=\Prob_H^{\mathcal{F}}(J_{F,g}).\]

We write the free boundary and $F$ boundary probabilities as follows:
\[p_g^\varnothing:=\Prob_G^{\mathcal{F}}(g\in A),\qquad p_g^F:=
\Prob_G^{F,\mathrm{bdry}}(g\in A).\]

Without boundary constraints, $g$ can be added
precisely on $D_g$. With boundary condition $F$, configurations
excluding $g$ must satisfy $B_F$, while configurations
including $g$ correspond to
\[\{A\cup F\cup\{g\}\text{ is a forest}\}
=(B_F\cap D_g)\setminus J_{F,g}.\]
Thus, separating configurations according to whether they contain
$g$ gives
\[p_g^\varnothing=\frac{z_gd}{1+z_gd},
\qquad p_g^F=\frac{z_g(c-j)}{b+z_g(c-j)}.\]
Subtracting the two inclusion probabilities, we obtain
\[
p_g^\varnothing-p_g^F=
\frac{z_g(j-\Cov_H^{\mathcal{F}}(\mathbf1_{B_F},\mathbf1_{D_g}))
}{(1+z_gd)\bigl(b+z_g(c-j)\bigr)}.\]
The upper covariance bound makes this difference
nonnegative. The lower covariance bound makes the numerator
at most $z_gj$. Since the denominator is at least $b>0$,
we deduce
\[0\leq p_g^\varnothing-p_g^F\leq z_g\frac{j}{b}=z_g\Prob_H^{\mathcal{F}}(J_{F,g}\mid B_F).\]

Let $\partial_F G$ be the set of boundary vertices incident
to an auxiliary edge of $F$. On $J_{F,g}$, the vertices
$x,y$ are disconnected in $A$ but connected in $A\cup F$.
The resulting path must use an edge of $F$. Following this
path from $x$ up to its first auxiliary edge shows that
$x$ is connected in $A$ to $\partial_F G$. Consequently,
\[
0\leq p_g^\varnothing-p_g^F\leq z_g\Prob_H^{\mathcal{F}}(J_{F,g}\mid B_F)\leq z_g\Prob_H^{F,\mathrm{bdry}}(x\leftrightarrow\partial_F G).\]
Thus the effect of the boundary condition on the inclusion
probability of $g$ is bounded by a single connection
probability, with no factor depending on the number of
boundary vertices.

Before proceeding with the theorem, we recall the contrasting percolation behaviour of the
arboreal gas in two and higher dimensions.
Bauerschmidt, Crawford, Helmuth, and Swan
\cite{bauerschmidt2021random} proved that, for every
$\lambda>0$, the arboreal gas on $\mathbb Z^2$ has no
infinite trees. In the augmented Gibbs formulation, this
holds for every infinite-volume Gibbs measure, including
limits obtained with arbitrary boundary conditions; see
\cite[Remark~1]{halberstam-hutchcroft-2024}.
Thus, for every such measure $\mu$,
\[
\mu(\text{every component of }A\text{ is finite})=1.
\]

In contrast, Bauerschmidt, Crawford, and Helmuth
\cite{bauerschmidt2024percolation} proved that, for every
$d\geq3$, infinite trees occur at sufficiently large
fugacity. More precisely, for sufficiently large
$\lambda$, subsequential limits of the arboreal gas on
the discrete tori considered in their work contain an
infinite tree almost surely.
The two-dimensional absence of infinite trees is the
input we use below to control the probability of
connecting a fixed vertex to a distant boundary.

\begin{thm}[Conditional Gibbs uniqueness in two dimensions]
\label{thm:conditional-gibbs-uniqueness}
Assume that the covariance bounds
\[
0\leq
\Cov_H^{\mathcal{F}}(\mathbf1_{B_F},\mathbf1_{D_g})
\leq
\Prob_H^{\mathcal{F}}(J_{F,g})
\]
hold for every finite graph $H$, every choice of positive
edge fugacities, every auxiliary boundary forest $F$,
and every further auxiliary edge $g$.
Then, for every $\lambda>0$, the arboreal gas measure with fugacity $\lambda$ on $\mathbb Z^2$ has a unique infinite-volume
Gibbs measure.
\end{thm}

\begin{proof}
Fix $\lambda>0$ and use the exhaustion
\[G_N=\mathbb Z^2[\{x\in\mathbb Z^2:\|x\|_1\leq N\}].\]
As discussed above, its boundary conditions can be
represented by auxiliary forests on $\partial G_N$.

By \cite[Proposition~2.1]{halberstam-hutchcroft-2024},
Gibbs measures are precisely subsequential limits of
finite-volume measures with possibly random boundary
conditions. Moreover, every such Gibbs measure on
$\mathbb Z^2$ has only finite trees almost surely by
\cite{bauerschmidt2021random}.

We first deduce that, for every fixed vertex $x$,
\[
\varepsilon_N(x):=\sup_F\Prob_{G_N}^{F,\mathrm{bdry}}
(x\leftrightarrow\partial G_N)\longrightarrow0,\]
where the supremum ranges over auxiliary forests representing
boundary partitions of $\partial G_N$.
Otherwise, there would be $\varepsilon>0$, a sequence
$N_j\to\infty$, and boundary conditions $F_j$ such that
\[\Prob_{G_{N_j}}^{F_j,\mathrm{bdry}}
(x\leftrightarrow\partial G_{N_j})\geq\varepsilon.\]
We pass to a subsequential local limit $\mu$, which is a Gibbs
measure. For each fixed sufficiently large $R$, connection
from $x$ to $\partial G_{N_j}$ implies connection from $x$
to $\partial G_R$ using edges of $G_R$. The latter event
depends on finitely many edges, so
\[\mu(x\leftrightarrow\partial G_R\text{ using edges of }G_R)\geq\varepsilon.\]
Letting $R\to\infty$ shows that the component of $x$ is
infinite with probability at least $\varepsilon$, a
contradiction.

Next, we fix a finite edge set $K$ and positive numbers
$\mathbf t=(t_h:h\in K)$, and replace the fugacity of
$h\in K$ by $\lambda t_h$, and write
\[R(\mathbf t)=\prod_{h\in K}\frac{\max\{1,t_h\}}{\min\{1,t_h\}}.\]
The resulting measure has density at most $R(\mathbf t)$
relative to the unmodified measure, uniformly in $N$
and the boundary condition.

Furthermore, under the unmodified measure,
\[\Prob_{G_N}^{F,\mathrm{bdry}}(g\notin A)\geq\frac1{1+\lambda}.\]
To see this, note that deleting $g$ maps every admissible configuration
containing $g$ injectively to an admissible configuration
excluding $g$, and divides its weight by $\lambda$.
The measure on $G_N-g$ is the measure on $G_N$ conditioned
on $g\notin A$. Hence,
\[
\Prob_{G_N-g,\mathbf t}^{F,\mathrm{bdry}}(x\leftrightarrow\partial G_N)\leq R(\mathbf t)(1+\lambda)\varepsilon_N(x)\longrightarrow0 \]
uniformly in $F$.

Now, fix an edge $g=xy$ in the interior of $G_N$, and let
$z_g$ be its possibly modified fugacity. Applying
the inequality $0\leq p_g^\varnothing-p_g^F\leq z_g\Prob_H^{\mathcal{F}}(J_{F,g}\mid B_F)
\leq z_g\Prob_H^{F,\mathrm{bdry}}(x\leftrightarrow\partial_F G)$ we obtained above to $H=G_N-g$ gives
\[0\leq
\Prob_{G_N,\mathbf t}^{\mathcal{F}}(g\in A)-\Prob_{G_N,\mathbf t}^{F,\mathrm{bdry}}(g\in A)\leq z_g\Prob_{G_N-g,\mathbf t}^{F,\mathrm{bdry}}(x\leftrightarrow\partial G_N).\]

Thus, the difference
tends to zero uniformly over boundary conditions.
The same conclusion holds for arbitrary mixtures of
boundary conditions.

It remains to show that this determines the full local
law. Let $\mu$ and $\nu$ be any two Gibbs measures, and
fix a finite edge set $K$. For $\mathbf t>0$, define
\[\Psi_\mu(\mathbf t)=\mathbb E_\mu\!\left[
\prod_{h\in K}t_h^{\mathbf1_{\{h\in A\}}}\right],\qquad
\frac{d\mu_{\mathbf t}}{d\mu}(A)=\frac{
\prod_{h\in K}t_h^{\mathbf1_{\{h\in A\}}}}{\Psi_\mu(\mathbf t)},\]
and define $\Psi_\nu$ and $\nu_{\mathbf t}$ similarly.

For every sufficiently large $N$, the augmented Gibbs
property expresses the restriction of $\mu_{\mathbf t}$
to $G_N$ as a mixture of finite-volume measures with the
modified fugacities. The same holds for $\nu_{\mathbf t}$.
Indeed, multiplying by the displayed density changes the
interior edge weights and may change the mixing
distribution of the boundary condition, but leaves the
form of the conditional interior law unchanged.

The uniform comparison with free boundary conditions
therefore implies, for every $g\in K$,
\[\mu_{\mathbf t}(g\in A)=\nu_{\mathbf t}(g\in A).\]
On the other hand,
\[t_g\frac{\partial}{\partial t_g}\log\Psi_\mu(\mathbf t)=
\mu_{\mathbf t}(g\in A),\]
and likewise for $\nu$. Thus all partial derivatives of
$\log\Psi_\mu-\log\Psi_\nu$ vanish on $(0,\infty)^K$.
Since
\[\Psi_\mu(\mathbf1)=\Psi_\nu(\mathbf1)=1,\]
we conclude that $\Psi_\mu=\Psi_\nu$.

The coefficients of these polynomials are exactly the
probabilities of the configurations on $K$, because
\[\Psi_\mu(\mathbf t)=\sum_{S\subseteq K}
\mu(A\cap K=S)\prod_{h\in S}t_h.\]
Hence $\mu$ and $\nu$ agree on every finite collection of
edges, and therefore $\mu=\nu$.
Existence follows from compactness and
\cite[Proposition~2.1]{halberstam-hutchcroft-2024}.
\end{proof}

\section{Matroids and Negative Correlation}
This is a speculative section where we explore the behavior of the covariance lower bound conjecture we have in the matroid setting. We believe that the conjecture is true in the matroid setting for appropriate events described below.  

Seymour and Welsh \cite{Seymour_Welsh_1975} showed that the negative correlation conjecture does not work in the matroid setting by exhibiting the $\mathcal{S}_{8}$ example. This matroid can be represented using the binary matrix 
\[
\begin{pmatrix}
1&1&1&1&1&1&1&0\\
0&1&0&0&0&1&1&1\\
0&0&1&0&1&0&1&1\\
0&0&0&1&1&1&0&1
\end{pmatrix}.
\]

For the pair $1$ and $8$, the counts of bases are 
\[
B=48,\qquad B_1=28,\qquad B_8=20,\qquad B_{18}=12,
\]
where $B_i$ denotes the number of bases containing the element $i$. We observe that the ordinary base-Rayleigh difference is \[B_1B_8-BB_{18}=-16.\]
Thus, the negative edge correlation fails for the above example. 

We can instead consider the connectivity version for Matroids. We note that in the graphic matroid, $e\in \operatorname{cl}(A)$ for an edge $e$ and an edge $A$, means that adding $e$ does not increase the matroid rank of the edge set $A$. Thus, the endpoints of $e$ are connected via the edges in $A$ already. Using this interpretation, we state the following conjecture. 

\begin{conj}
\label{conj: matroidcorrelation}
Let $M=(E,\mathcal{B})$ be a finite matroid with ground set $E$, and let $e,f\in E$ be distint elements. Let $H=M-e-f$. Suppose we have weights assigned to each element $g\in E$, $z_g \in \mathbb{R}_{\geq 0}$. Sample an independent set $A$ of $H$ with probability proportional to the product of its positive weight elements. Then, we have 
\[\Cov(\mathbf{1}_{e\in \operatorname{cl}_M(A)},\mathbf{1}_{f\in \operatorname{cl}_M(A)})\geq 0.\]
\end{conj}

Now, consider the following partition functions on each closure event where $\mathcal{I}(H)$ denotes the collection of independent sets of $H$, 
\[I_{ij}(\mathbf{z})
=\sum_{\substack{A\in\mathcal I(H)\\
        \1_{\{e\in\cl_M(A)\}}=i,\ 
        \1_{\{f\in\cl_M(A)\}}=j}}
  \prod_{g\in A}z_g,
\qquad i,j\in\{0,1\}.\]

These corresponds to the partition functions of the events that $\{a\leftrightarrow b\}$, $\{c\leftrightarrow d\}$ and their complements in graph setting where $e=ab$ and $f=cd$. 

With these, the numerator when we open the covariance becomes 
\[\Gamma_{ef}^M(\mathbf{z}):=I_{11}(\mathbf{z})I_{00}(\mathbf{z})-I_{10}(\mathbf{z})I_{01}(\mathbf{z}).\]

Non-negativity of every monomial coefficient in $\Gamma_{ef}^M(\mathbf{z})$ would thus imply Conjecture \ref{conj: matroidcorrelation} for all positive element weights. We tested the non-negativity of these coefficients using computers. Table \ref{tab:matroid-experiments} summarizes the results. The finite geometries are the Fano matroid, $\operatorname{PG}(2,3)$, $\operatorname{PG}(3,2)$, $\operatorname{AG}(2,3)$, and $\operatorname{AG}(3,2)$, together with their duals.  The random representable matroids were split equally among $\mathrm{GF}(2)$, $\mathrm{GF}(3)$, and $\mathrm{GF}(5)$.  For each random matroid and its dual we used eight marked pairs; for the uniform matroids and finite geometries, where the pairs are equivalent, we used one; and for $\mathcal S_8$, $V_8$, and their duals we used every pair.

\begin{center}
\begin{minipage}{0.86\textwidth}
\captionsetup{hypcap=false}
\captionof{table}{Exact coefficientwise tests of $\Gamma^{M}_{ef}(\mathbf{z})$.}
\label{tab:matroid-experiments}
\centering
\begin{tabular}{lrr}
\toprule
Family & Matroids & Marked pairs\\
\midrule
Uniform $U_{r,n}$, $2\leq n\leq12$ & $88$ & $88$\\
Finite geometries and duals & $10$ & $10$\\
$\mathcal S_8$ and $\mathcal S_8^*$ & $2$ & $56$\\
V\'amos $V_8$ and $V_8^*$ & $2$ & $56$\\
Random representable matroids and duals & $180$ & $1{,}440$\\
Random sparse-paving matroids and duals & $180$ & $1{,}440$\\
Random transversal matroids and duals & $120$ & $960$\\
\midrule
Total & $582$ & $4{,}050$\\
\bottomrule
\end{tabular}
\end{minipage}
\end{center}

Across these tests, $2{,}853{,}781$ monomial coefficients were strictly positive and none was negative.  The presentations, random seed, and exact computation are recorded in the companion file \texttt{explore\_matroid\_closure\_correlation.py}.

One important test here is the matroid $\mathcal{S}_8$. We observe that for $\mathcal{S}_8$, every coefficient of $\Gamma_{ef}^{\mathcal{S}_8}$ is non-negative for any pair of $e,f$ given.  

\section{Conclusion and Further Questions}

In this paper, we showed negative edge correlation for the arboreal gas measure at large $\lambda$-fugacity and for the random-cluster measure with $0<q<1$ with equal edge weights $z\geq 1$ on the complete graph $K_n$. One important aspect was the fact that rather than showing positivity for values, we showed that the polynomial itself splits as sums of squares for the arboreal gas measure. A direct continuation of this result would be to use similar ideas to establish positivity for the third and the later coefficients as well. Our example \ref{ex:k4-chain} shows that we can get the first $k$ highest coefficients to be all zero. 

We list the main conjectures, some that motivated this paper here, and some that arose during the course of this work again, since the reader who has gone through this paper might now have a new perspective on them. First, we have the original conjecture of negative edge correlation for forests and the random-cluster measure which was the motivation of this paper itself. In full generality, this conjecture remains open. 
\medskip

\noindent\textbf{Conjecture~\ref{conj:kahn} (restated).}
For every finite graph $G$, every $\lambda>0$, and all distinct edges $e$ and $f$,
\[
\Rayleigh_{ef}(\ForestZ_G)(\lambda\mathbf1)\geq0.
\]

\medskip
\noindent\textbf{Conjecture~\ref{conj:grimmett-winkler} (restated).}
Let $0<q<1$ and let all edge weights be nonnegative.  Then, for every finite graph $G$ and distinct edges $e,f$,
\[
\Rayleigh_{ef}(\RCZ_G)(q,\mathbf z)\geq0.
\]

Our result is good initial evidence for the sums of squares conjecture of Wagner. 
\medskip

\noindent\textbf{Conjecture~\ref{conj:sumsofsquares} (restated).}
Let $G=(V,E)$ be a finite graph, and let $e,f\in E$ be distinct edges. The Rayleigh difference $\Rayleigh_{ef}(\ForestZ_G)$ can be expressed as a finite sum of monomials in the edge weights multiplied by squares of polynomials in those weights. See Wagner~\cite{wagner2008negatively} for the more precise combinatorial formulation.

In Section \ref{sec:connectivity-programme}, we started considering the problem from a different perspective, the vertex connectivities. Here, we conjecture that the events $a\leftrightarrow b$ and $c\leftrightarrow d$ are positively correlated both in the arboreal gas measure and the random-cluster measure with $0<q<1$.

\medskip
\noindent\textbf{Conjecture~\ref{conj: forest-connectivity} (restated).}
Let $G=(V,E)$ be a finite graph, and let $a,b,c,d\in V$ be distinct vertices. Under the arboreal gas measure with positive fugacity at each edge, we have
\[\Gamma_G^{\mathcal{F}}(a,b;c,d) \geq 0.\]
In other words, 
\[\Prob_G^{\mathcal{F}}(a\leftrightarrow b, c \leftrightarrow d)\geq \Prob_G^{\mathcal{F}}(a\leftrightarrow b)\Prob_G^{\mathcal{F}}(c \leftrightarrow d).\]

\medskip
\noindent\textbf{Conjecture~\ref{conj: rc-connectivity} (restated).}
Let $G=(V,E)$ be a finite graph, and let $a,b,c,d\in V$ be distinct vertices. For every $0<q<1$, under the random-cluster measure, we have
\[\Gamma_G^{\mathrm{RC}}(a,b;c,d) \geq 0.\]
In other words, 
\[\Prob_G^{\mathrm{RC}}(a\leftrightarrow b, c \leftrightarrow d)\geq \Prob_G^{\mathrm{RC}}(a\leftrightarrow b)\Prob_G^{\mathrm{RC}}(c \leftrightarrow d).\]

\medskip
\noindent\textbf{Acknowledgements.}
The authors thank Arvind Ayyer and Svante Linusson for helpful discussions.  The research of the first author is supported by Swedish Research Council grant VR 2021-03687.  The second author thanks Zessta Software Services for supporting this independent research.
\begin{comment}
\appendix
\section{Appendix : Constant fugacity random forests in \texorpdfstring{$K_n$}{Kn}}

In this appendix, we include an AI generated proof (this proof was generated by GPT 6 Astra) of negative edge correlation for random forests (all edges have constant weight $\lambda$ in $K_n$). 

Fix $K_n$ and $\lambda$. Denote $\mathbb{P}[e \in F]$ by $x = x(\lambda, n)$. Negative edge correlation can be rewritten purely in terms of $\lambda$. 

\begin{prop}
Two adjacent edges are negatively correlated if and only if 
\[\Psi(x, n, \lambda) := (n-2)x^2 - (2+n\lambda)x + 2\lambda \geq 0.\]
And two non-adjacent edges are negatively correlated if and only if 
\[\Phi(x, n, \lambda) := \lambda(n-2)(n-3)x^2 - [n(n-4)\lambda^2 + (2n-9)\lambda -1]x - \lambda + 2(n-4)\lambda^2 \geq 0.\]
\end{prop}

We can get precise control over $x$ using the following fact. 

\begin{prop}
The marginal $x$ satisfies the second order Ricatti equation, 
\[2\lambda^2 x' = -n(n-1)\lambda x^2 + [n^2\lambda^2 + (2n+1)\lambda -1]x -2n\lambda^2 + \lambda.\]
\end{prop}

These two facts can be used to give a compact proof of negative correlation for large $n$. Indeed, standard techniques can be used to show that 
\[x = \dfrac{2}{n} + O\left(\dfrac{1}{\lambda n^2}\right).\]
This in turn shows that 
\[\Phi(x, n, \lambda) = \dfrac{4}{n^2} + O\left(\dfrac{1}{n^2}\right).\]

\end{comment}

\begingroup
\small
\bibliographystyle{amsplain}
\bibliography{jv}
\endgroup

\end{document}